\documentclass[a4paper,11pt]{article}
\usepackage[utf8]{inputenc}
\usepackage[T1]{fontenc}
\usepackage[english]{babel}
\usepackage[margin=2.8cm]{geometry}
\usepackage{mathpazo}
\usepackage[scaled=.95]{helvet}
\usepackage{courier}
\usepackage{amsmath,amsfonts,amssymb,amsthm}
\usepackage{graphicx}
\usepackage{float}
\usepackage{natbib}
\usepackage{color}
\usepackage[onehalfspacing]{setspace}

\newtheorem{prop}{Proposition}[section]

\newtheorem{rmq}{Remark}[section]  
\newtheorem{theo}{Theorem}[section]  
\newtheorem{lem}{Lemma}[section]

\newtheorem{cor}{Corollary}[section]
\DeclareUnicodeCharacter{B0}{\textdegree}

\usepackage[draft]{fixme}
\fxusetheme{color}
\FXRegisterAuthor{h}{ah}{H}
\FXRegisterAuthor{p}{ap}{P}
\FXRegisterAuthor{a}{aa}{A}

\usepackage{hyperref}
\hypersetup{
pdfpagemode=UseOutlines,      % UseOutlines, UseThumbs, None, FullScreen : agencement au démarrage
pdfstartview=Fit,             % Fit, FitH, FitB, FitBH : vue de la page au départ (pleine largeur...)
pdffitwindow=true,            % bool: Maximiser
pdfpagelayout=TwoColumnsRight,% SinglePage, TwoColumnsRight/Left, OneColumn : affichage des pages
pdftoolbar=true,              % bool: Affichage de la barre d'outils
pdfmenubar=true,              % bool: Affichage de la barre de menus
bookmarksopen=false,          % bool: Dépliement des signets
bookmarksnumbered=true,       % bool: Numérotation des signets
colorlinks=true,              % bool: Liens colorés
pdfauthor={Ton nom},     % Auteur
pdftitle={Titre PDF},    % Titre
pdfcreator=PDFLaTeX,          % 
pdfproducer=PDFLaTeX,         %
linkcolor=blue,               % Couleur des liens
urlcolor=blue,                %             url
anchorcolor=blue,            %         du texte
citecolor=blue,              % Couleur de citation 
frenchlinks=true,             % bool: Utiliser des petites majuscules pour les liens, plutôt que de la couleur
pdfborder={0 0 0}             % Ne pas encadrer les liens
}

\begin{document}
\title{Online estimation of the asymptotic variance for averaged stochastic gradient algorithms}
\author{Antoine Godichon-Baggioni \\ Institut de Math\'ematiques de Toulouse, \\
Université Paul Sabatier, 31000 Toulouse, France \\
email: godichon@insa-toulouse.fr
} 
\maketitle

\begin{abstract}
Stochastic gradient algorithms are more and more studied since they can deal efficiently and online with large samples in high dimensional spaces. In this paper, we first establish a Central Limit Theorem for these estimates as well as for their averaged version in general Hilbert spaces. Moreover, since having the asymptotic normality of estimates is often unusable without an estimation of the asymptotic variance, we introduce a new recursive algorithm for estimating this last one, and we establish its almost sure rate of convergence as well as its rate of convergence in quadratic mean. Finally, two examples consisting in estimating the parameters of the logistic regression and estimating geometric quantiles are given.
\end{abstract}


\section{Introduction}

High Dimensional and Functional Data Analysis are interesting domains which do not have stopped growing for many years. To consider these kinds of data, it is more and more important to think about methods which take into account the high dimension as well as the possibility of having large samples. In this paper, we focus on an usual stochastic optimization problem which consists in estimating
\[
m:= \arg \min_{h\in H} \mathbb{E}\left[ g \left( X , h \right) \right] ,
\]
where $X$ is a random variable taking values in a space $\mathcal{X}$ and $g:\mathcal{X}\times H \longrightarrow  \mathbb{R}$, where $H$ is a separable Hilbert space. In order to build an estimator of $m$, an usual method was to consider the solver of the problem generated by the sample, i.e to consider $M$-estimates (see \cite{HubR2009} and \cite{MR2238141} among others). In order to build these estimates, deterministic convex optimization algorithms (see \cite{boyd2004convex}) are often used (see \cite{VZ00}, \cite{oja1985asymptotic} in the case of the median), and these methods are really efficient in small dimensional spaces. 

\medskip

Nevertheless, in a context of high dimensional spaces, this kind of method can encounter many computational problems. The main ones are that it needs to store all the data, which can be expensive in term of memory and that they cannot deal online with the data. In order to overcome this, stochastic gradient algorithms (\cite{robbins1951}) are efficient candidates since they do not need to store the data into memory, and they can be easily updated, which is crucial if the data arrive sequentially (see \cite{dufloalgo}, \cite{Duf97}, \cite{Kushner2003} or \cite{nemirovski2009robust} among others). In order to improve the convergence, \cite{ruppert1988efficient} and \cite{PolyakJud92} introduced its averaged version (see also \cite{dippon1997weighted} for a weighted version). These algorithms have become crucial to statistics and modern machine learning (\cite{bach2013non}, \cite{bach2014adaptivity}, \cite{juditsky2014deterministic}). There are already many results on these algorithms in the literature, that we can split into two parts: asymptotic results, such as almost sure rates of convergence \citep{schwabe1996stochastic,Duf97,walk1992foundations,pelletier1998almost,Pel00}, 
 and non asymptotic ones, such as rates of convergence in quadratic mean \citep{CCG2015,godichon2015,bach2013non,bach2014adaptivity,nemirovski2009robust}. 

\medskip

In a recent work, \cite{godichon2016} introduces a new framework, with only locally strongly convexity assumptions, in general Hilbert spaces, which allows to obtain almost sure and $L^{p}$ rates of convergence. In keeping with it, and in order to have a deeper study of the stochastic gradient algorithm as well as of its averaged version (up to a new assumption), we first give the asymptotic normality of the estimates. In a second time, since a Central Limit Theorem is often unusable without an estimation of the variance, we introduce a recursive algorithm, inspired by \cite{gahbiche2000estimation}, to estimate the asymptotic variance of the averaged estimator and we establish its rates of convergence. As far as we know, there was not yet an efficient and recursive estimate of the asymptotic variance in the literature. Finally, two  examples of application are given. The first usual one consists in estimating the parameters of the logistic regression \citep{bach2014adaptivity} while the second one consists in estimating geometric quantiles (see \cite{Cha96} and \cite{ChaCha2014}), which are useful robust indicators in statistics. Indeed, they are often used in data depth and outliers detection (\cite{serfling2006depth}, \cite{hallin2006semiparametrically}), as well as for robust estimation of the mean and variance (see \cite{minsker2014robust}), or for Robust Principal Component Analysis (\cite{Ger08}, \cite{KrausPanaretos2012}, \cite{CG2015}).

\medskip

The paper is organized as follows: Section \ref{sectionassumption} recalls the framework introduced by \cite{godichon2016} before giving two new assumptions which allow to get the rate of convergence of the estimators of the asymptotic variance. In section \ref{sectiontlc}, the stochastic gradient algorithm as well as its averaged version are introduced and their asymptotic normality are given. The recursive estimator of the asymptotic variance is given in Section \ref{sectionvariance} and its almost sure as well as its quadratic mean rates of convergence are established. Applications, consisting in estimating the logistic regression parameters and in the recursive estimation of geometric quantiles, are given  in Section \ref{sectionapplication} as well as a short simulation study. Finally, the proofs are postponed in Section \ref{sectionproof} and in Appendix.

\section{Assumptions}\label{sectionassumption}
Let $H$ be a separable Hilbert space such as $\mathbb{R}^{d}$ or $L^{2}(I)$ (for some closed interval $I \subset \mathbb{R}$), we denote by $\left\langle .,. \right\rangle$ its inner product and by $\left\| . \right\|$ the associated norm. Let $X$ be a random variable taking values in a space $\mathcal{X}$, and let $G: H \longrightarrow \mathbb{R}$ be the function we would like to minimize, defined for all $h \in H$ by
\begin{equation}
G (h) := \mathbb{E}\left[ g(X,h ) \right],
\end{equation}
where $g: \mathcal{X} \times H \longrightarrow \mathbb{R}$. Moreover, let us suppose that the functional $G$ is convex. Finally, let us introduce the space of linear operators on $H$, denoted by $\mathcal{S}(H)$, equipped with the Frobenius (or Hilbert-Schmidt) inner product, which is defined by
\[
\left\langle A,B \right\rangle_{F} := \sum_{j \in J} \left\langle A (e_{j} ) , B ( e_{j} ) \right\rangle , \quad \forall A,B \in \mathcal{S}(H) ,
\]
where $\left( e_{j} \right)_{j \in J}$ is an orthonormal basis of $H$. We denote by $\left\| . \right\|_{F}$ the associated norm, and $\mathcal{S}(H)$ is then a separable Hilbert space. Let us recall the framework introduced by \cite{godichon2016}:
\begin{itemize}
\item[\textbf{(A1)}] The functional $g$ is Frechet-differentiable for the second variable almost everywhere. Moreover, $G$ is differentiable and there exists $m \in H$ such that
\[
 \nabla G (m) = 0 .
\]
\item[\textbf{(A2)}] The functional $G$ is twice continuously differentiable almost everywhere and for all positive constant $A$, there is a positive constant $C_{A}$ such that for all $h \in \mathcal{B}\left( m , A \right)$,
\[
\left\| \Gamma_{h} \right\|_{op} \leq C_{A} ,
\]
where $\Gamma_{h}$ is the Hessian of the functional $G$ at $h$ and $ \left\| . \right\|_{op}$ is the usual spectral norm for linear operators.
\item[\textbf{(A3)}] There exists a positive constant $\epsilon$ such that for all $h \in \mathcal{B}\left( m , \epsilon \right)$, there is an orthonormal basis of $H$ composed of eigenvectors of $\Gamma_{h}$. Moreover,  let us denote by $\lambda_{\min}$ the limit inf of the eigenvalues of $\Gamma_{m}$, then $\lambda_{\min}$ is positive. Finally, for all $h \in \mathcal{B}\left( m , \epsilon \right)$, and for all eigenvalue $\lambda_{h}$ of $\Gamma_{h}$, we have $\lambda_{h} \geq \frac{\lambda_{\min}}{2} > 0$.
\item[\textbf{(A4)}] There are positive constants $\epsilon, C_{\epsilon}$ such that for all $h \in \mathcal{B}\left( m , \epsilon \right)$,
\[
\left\| \nabla G (h) - \Gamma_{m}(h-m) \right\| \leq C_{\epsilon} \left\| h-m \right\|^{2}.
\]
\item[\textbf{(A5)}] 
\begin{itemize}
\item[\textbf{(a)}] There is a positive constant $L_{1}$ such that for all $h \in H$,
\[
\mathbb{E}\left[ \left\| \nabla_{h}g \left( X,h \right) \right\|^{2} \right] \leq L_{1} \left( 1 + \left\| h-m \right\|^{2} \right) . 
\]
\item[\textbf{(a')}] There is a positive constant $L_{2}$ such that for all $h \in H$,
\[
\mathbb{E}\left[ \left\| \nabla_{h}g \left( X,h \right) \right\|^{4} \right] \leq L_{2} \left( 1 + \left\| h-m \right\|^{4} \right) . 
\]
\item[\textbf{(b)}] For all integer $q$, there is a positive constant $L_{q}$ such that for all $h \in H$,
\[
\mathbb{E}\left[ \left\| \nabla_{h}g \left( X,h \right) \right\|^{2q} \right] \leq L_{q} \left( 1 + \left\| h-m \right\|^{2q} \right) . 
\]
\end{itemize}
\end{itemize}
Let us now make some comments on assumptions. First, Assumption \textbf{(A1)} ensures the existence of a solution and enables to use a stochastic gradient descent, while \textbf{(A2)} gives some smoothness properties on the objective function. Assumption \textbf{(A3)} ensures the uniqueness of the minimizer of $G$, and \textbf{(A4)},\textbf{(A5)} give bounds of the gradient and of the remainder term of its Taylor's expansion. The main difference between this framework and the usual one for strongly convex objective is that we just assume the local strong convexity of the objective function, and in return, $p$-th moments of the gradient of the functional $g$ have to be bounded. Note also that the Hessian of the functional $G$ is not supposed to be compact, so that its smallest eigenvalue does not necessarily converge to $0$ when the dimension tends to infinity (a counter example is given in Section \ref{sectionapplication}).  Remark that assumptions \textbf{(A1)} to \textbf{(A5b)} are deeply discussed in \cite{godichon2016}. Let us now introduce two new assumptions.
\begin{itemize}
\item[\textbf{(A6)}] Let $\varphi : H \longrightarrow \mathcal{S}(H)$ be the functional defined for all $h \in H$ by
\[
\varphi \left( h \right) := \mathbb{E}\left[ \nabla_{h} g \left( X , h \right) \otimes \nabla_{h}g \left( X,h \right) \right] .
\]
\begin{itemize}
\item[\textbf{(a)}] The functional $\varphi$ is continuous at $m$ with respect to the Frobenius norm:
\[
\lim_{h \to m} \left\| \mathbb{E}\left[ \nabla_{h} g \left( X , m \right) \otimes \nabla_{h}g \left( X,m \right) \right] - \mathbb{E}\left[ \nabla_{h} g \left( X , h \right) \otimes \nabla_{h}g \left( X,h \right) \right] \right\|_{F} = 0 .
\]
\item[\textbf{(b)}] The functional $\varphi$ is locally lipschitz on a neighborhood of $m$: there are positive constants $\epsilon, C_{\epsilon}'$, such that for all $h \in \mathcal{B}\left( m , \epsilon \right) $,
\[
\left\| \mathbb{E}\left[ \nabla_{h} g \left( X , m \right)\otimes \nabla_{h} g \left( X , m \right) - \nabla_{h} g \left( X, h \right) \otimes \nabla_{h} g \left( X , h \right) \right] \right\|_{F} \leq C_{\epsilon}' \left\| h-m \right\| .
\]
\end{itemize}
\end{itemize}
Assumption \textbf{(A6a)} enables to establish the asymptotic normality of the stochastic gradient descent as well as of its averaged version. Note that under \textbf{(A5a)}, the functional $\varphi$ is bounded, and more precisely
\[
\left\| \mathbb{E}\left[ \nabla_{h}g \left( X,h \right)\otimes \nabla_{h}g \left( X,h \right) \right] \right\|_{F} \leq \mathbb{E}\left[ \left\| \nabla_{h} g \left( X,h \right) \right\|^2 \right] \leq L_{1} \left( 1+ \left\| h-m \right\|^{2} \right) . 
\]  
Assumption \textbf{(A6b)} can be verified by giving a bound, on a neighborhood of $m$, of the derivative of the functional $\varphi$. This last assumption allows to give the rate of convergence of the estimators of the asymptotic variance. An example is given for the special case of the geometric median in Appendix.
\begin{rmq}
For all $h\in H$ and $A>0$,
\[
\mathcal{B}\left( h , A \right) = \left\lbrace h' \in H, \quad \left\| h-h' \right\| < A \right\rbrace .
\]
\end{rmq}

\begin{rmq}
Let $h,h' \in H$, the linear operator $h \otimes h' : H \longrightarrow H $ is defined for all $h'' \in H$ by $h \otimes h' (h'') := \left\langle h,h'' \right\rangle h'$. Moreover,
\begin{equation}\label{normf}
\left\| h \otimes h' \right\|_{F} = \left\| h \right\| \left\| h' \right\| .
\end{equation}
\end{rmq}

\section{The stochastic gradient algorithm and its averaged version}\label{sectiontlc}
\subsection{The Robbins-Monro algorithm}
In what follows, let $X_{1},...,X_{n}$ be independent random variables with the same law as $X$. The stochastic gradient algorithm is defined recursively for all $n \geq 1$ by
\begin{equation}
\label{defirm} m_{n+1} = m_{n} - \gamma_{n} \nabla_{h}g \left( X_{n+1} , m_{n} \right),
\end{equation}
with $m_{1}$ bounded and $\left( \gamma_{n} \right)$ is a step sequence of the form $\gamma_{n} := c_{\gamma} n^{-\alpha}$, with $c_{\gamma}>0$ and $\alpha \in \left( \frac{1}{2}, 1 \right)$. Moreover, let $\left( \mathcal{F}_{n} \right)_{n \geq 1}$ be the sequence of $\sigma$-algebras defined for all $n \geq 1$ by $\mathcal{F}_{n} := \sigma \left( X_{1},...,X_{n} \right)$. Then, the algorithm can be considered as a noisy (or stochastic) gradient algorithm since it can be written as
\begin{equation}
\label{decphi} m_{n+1} = m_{n} - \gamma_{n} \Phi \left( m_{n} \right) + \gamma_{n}\xi_{n+1},
\end{equation} 
where $\Phi \left( m_{n} \right) := \nabla G \left( m_{n} \right)$, and $\left( \xi_{n} \right)$, defined for all $n \geq 1$ by $\xi_{n+1} := \Phi \left( m_{n} \right) - \nabla_{h} g \left( X_{n+1} , m_{n} \right)$, is a martingale differences sequence adapted to the filtration $\left( \mathcal{F}_{n} \right)$.  Finally, note that under assumptions \textbf{(A1)} to \textbf{(A5a)}, it was proven in \cite{godichon2016} that for all positive constant $\delta$,
\begin{equation}
\label{vitesseasrm} \left\| m_{n} - m \right\|^{2} = o \left( \frac{\left( \ln n \right)^{\delta}}{n^{\alpha}}\right) \quad a.s.
\end{equation}
Moreover, assuming that \textbf{(A5b)} is also fulfilled, for all positive integer $p$, there is a constant $C_{p}$ such that for all $n \geq 1$,
\begin{equation}
\label{vitlprm} \mathbb{E}\left[ \left\| m_{n} - m \right\|^{2p} \right] \leq \frac{C_{p}}{n^{p\alpha}}.
\end{equation}
In order to get a deeper study of this estimate, we now give its asymptotic normality.
\begin{theo}\label{tlcgrad}
Suppose assumptions \textbf{(A1)} to \textbf{(A5a')} and \textbf{(A6a)} hold. Then, we have the convergence in law
\[
\lim_{n \to \infty} \frac{1}{\sqrt{\gamma_{n}}}\left( m_{n} - m \right) \sim \mathcal{N} \left( 0 , \Sigma_{RM} \right),
\]
with
\begin{align*}
& \Sigma_{RM}  := \int_{0}^{+ \infty}e^{-s\Gamma_{m}} \Sigma ' e^{-s\Gamma_{m}} ds , \quad \quad \text{and} \quad \quad  \Sigma '  := \mathbb{E}\left[ \nabla_{h} g \left( X , m \right) \otimes \nabla_{h}g \left( X,m \right) \right] .
\end{align*}
\end{theo}
\noindent The proof is given in Appendix. Note that the variance $\Sigma_{RM}$ does not depend on the step sequence $\left( \gamma_{n} \right)$, but Theorem \ref{tlcgrad} could be written as
\[
\lim_{n \to \infty} n^{\alpha /2} \left( m_{n} - m \right) \sim \mathcal{N} \left( 0 , c_{\gamma}\Sigma_{RM} \right),
\]
\begin{rmq}
Let $M$ be a squared matrix, $e^{M}$ is defined by (see \cite{horn2012matrix} among others)
\[
e^{M} = \sum_{k=0}^{\infty} \frac{1}{k!}M^{k}.
\]
Thanks to assumptions \textbf{(A2)},\textbf{(A3)}, $0 < \lambda_{\min}\left( \Gamma_{m} \right) \leq \lambda_{\max} \left( \Gamma_{m} \right) < \infty$, while under \textbf{(A5a)} and by dominated convergence, 
\begin{align*}
\left\| \Sigma_{RM} \right\|_{F} & \leq \int_{0}^{+\infty}\left\| e^{-s\Gamma_{m}} \right\|_{op}^{2}\left\| \Sigma ' \right\|_{F}ds \leq \int_{0}^{+\infty} e^{-2s\lambda_{\min}}\left\| \Sigma ' \right\|_{F}ds   \leq \frac{L_{1}}{2\lambda_{\min}},
\end{align*} 
and $\Sigma_{RM}$ is so well defined.
\end{rmq}
\begin{rmq}
Note that analogous results are given by \citep{fabian1968asymptotic,pelletier1998almost} in the particular case of finite dimensional spaces while, for analogous results in Banach and Hilbert spaces, one can also see \cite{walk1992foundations}, \cite{ljung2012stochastic}, \cite{KY03}. 
\end{rmq}
\begin{rmq}\label{remlyap}
Note that taking a step sequence of the form $\gamma_{n} = \frac{c}{n}$ with $c > \frac{2}{\lambda_{\min}}$ is possible, and one can obtain the following asymptotic normality (see \cite{Pel00} among others for the case of finite dimensional spaces)
\[
\lim_{n \to \infty} \sqrt{n} \left( m_{n} - m \right) \sim \mathcal{N} \left( 0 , c\Sigma ' \right) .
\]
Nevertheless, it does not only necessitate to have some information on the Hessian $\Gamma_{m}$, but $c\Sigma '$ is also not the optimal variance (see \cite{Duf97} and \cite{Pel00} for instance).
\end{rmq}

\subsection{The averaged algorithm}
As mentioned in Remark \ref{remlyap}, having the parametric rate  of convergence ($O\left(\frac{1}{n}\right)$) with the Robbins-Monro algorithm is possible taking a good choice of step sequence $\left( \gamma_{n} \right)$. Nevertheless, this choice is often complicated and the asymptotic variance which is obtained is not optimal. Then, in order to improve the convergence, let us now introduce the averaged algorithm (see \cite{ruppert1988efficient} and \cite{PolyakJud92}) defined for all $n \geq 1$ by 
\[
\overline{m}_{n} = \frac{1}{n}\sum_{k=1}^{n} m_{k}.
\]
This can be written recursively for all $n \geq 1$ as
\begin{equation}
\overline{m}_{n+1} = \overline{m}_{n} + \frac{1}{n+1}\left( m_{n+1} - \overline{m}_{n} \right).
\end{equation}
It was proven in \cite{godichon2016} that under assumptions \textbf{(A1)} to \textbf{(A5a)}, for all $\delta > 0$,
\begin{equation}
\label{vitasmoy} \left\| \overline{m}_{n} - m \right\|^{2} = o \left( \frac{(\ln n)^{1+\delta}}{n}\right) \quad a.s.
\end{equation}
Suppose assumption \textbf{(A5b)} is also fulfilled, for all positive integer $p$, there is a positive constant $C_{p}'$ such that for all $n \geq 1$,
\begin{equation}\label{vitlpmoy}
\mathbb{E}\left[ \left\| \overline{m}_{n} - m \right\|^{2p} \right] \leq \frac{C_{p}'}{n^{p}}.
\end{equation}
Finally, in order to have a deeper study of this estimate, we now give its asymptotic normality.
\begin{theo}\label{theotlc}
Suppose assumptions \textbf{(A1)} to \textbf{(A5a')} and \textbf{(A6a)} are verified. Then, we have the convergence in law
\[
\lim_{n \to \infty} \sqrt{n} \left( \overline{m}_{n} - m \right) \sim \mathcal{N}\left( 0 , \Sigma \right),
\]
with $\Sigma := \Gamma_{m}^{-1}\Sigma ' \Gamma_{m}^{-1}$, and $\Sigma ' := \mathbb{E}\left[ \nabla_{h} g \left( X,m \right) \otimes \nabla_{h}g \left( X , m \right) \right]$.
\end{theo} 
The proof is given in Section \ref{sectionproof}. For analogous results, one can also see \cite{schwabe1996stochastic}, \cite{Pel00},  \cite{dippon2006averaged}.

\section{Recursive estimation of the asymptotic variance}\label{sectionvariance}

\subsection{Some existing estimators}
A first naive method to estimate the asymptotic variance could be to estimate the Hessian $\Gamma_{m}$ and the variance $\Sigma '$ as follows
\begin{align*}
\Gamma_{m}^{(n+1)} & = \Gamma_{m}^{(n)} + \frac{1}{n+1}\left( \nabla_{h}^{2}g \left( X_{n+1} , \overline{m}_{n} \right)  - \Gamma_{m}^{(n)} \right) , \\
\Sigma_{n+1}' &  = \Sigma_{n}' + \frac{1}{n+1}\left( \nabla_{h}g \left( X_{n+1} , \overline{m}_{n} \right) \otimes \nabla_{h}g \left( X_{n+1} , \overline{m}_{n} \right)  - \Sigma_{n}' \right) ,
\end{align*}
but the main problem is that under assumptions \textbf{(A2)}, \textbf{(A3)} and \textbf{(A5a)}, if $H$ is an infinite dimensional space, then
\[
\left\| \Gamma_{m} \right\|_{F} = \infty , \quad \quad \text{while} \quad \quad \left\| \Gamma_{m}^{-1}\Sigma ' \Gamma_{m}^{-1} \right\|_{F} \leq \frac{L_{1}}{\lambda_{\min}^2}.
\]
Another problem is that, in order to get a recursive estimator of the asymptotic variance, it needs to invert a matrix at each iteration, which costs much calculus time in high dimensional spaces. A second estimator of the asymptotic variance was introduced in \cite{Pel00}, defined for all $n \geq 1$ by
\begin{equation}
\widehat{\Sigma}_{n} = \frac{1}{\ln n}\sum_{k=1}^{n} \left( m_{k} - \overline{m}_{n} \right) \otimes \left( m_{k} - \overline{m}_{n} \right) ,
\end{equation}
and under \textbf{(A1)} to \textbf{(A6b)}, 
\[
\mathbb{E}\left[ \left\| \widehat{\Sigma}_{n} - \Sigma \right\|_{F}^{2} \right] = O \left( \frac{1}{\ln n}\right) .
\]
Thus, this estimator faces two main problems: it is not recursive and it converges very slowly. Finally, in order to solve the second problem, a faster algorithm was introduced by \cite{gahbiche2000estimation}, defined for all $n \geq 1$ by 
\begin{equation}
\label{algopel} \tilde{\Sigma}_{n} := \frac{1-\delta}{n^{1-\delta}}\sum_{k=1}^{n}\frac{1}{k^{\delta+s+\mu}}\exp \left( -\frac{k^{1-s}}{1-s} \right) \left( \sum_{j=1}^{k}j^{\mu /2}e^{\frac{j^{1-s}}{2(1-s)}}\left( m_{j} - \overline{m}_{n} \right) \right) \otimes \left( \sum_{j=1}^{k}j^{\mu /2}e^{\frac{j^{1-s}}{2(1-s)}}\left( m_{j} - \overline{m}_{n} \right) \right) ,
\end{equation}
with $(1+\alpha)/2 < s < 1$, $\mu \geq 0$ and $s/2 < \delta < (1+s)/2$. This algorithm is first based on an usual decomposition of the stochastic gradient algorithm (see equation (\ref{decdeltabis})) which enables to make appear a martingale term which carries the convergence rate (see equation (\ref{majbourrinsigmanbarre})). In a second time, the objective is to find step sequences which enable to improve the rate of convergence of the variance estimate (see \cite{gahbiche2000estimation} for technical details on assumptions on the step sequences). In the case of finite dimensional spaces, the following convergence in probability is given (under some assumptions)
\[
\frac{n^{1/2-s/2}}{\left( \ln \ln n \right)^{c}}\left\| \tilde{\Sigma}_{n} - \Sigma \right\|_{op} \xrightarrow[n\to \infty]{\mathbb{P}} 0 ,
\]
with $c>0$. A first technical problem is that only the convergence in probability is given, in the case of finite dimensional spaces, and for the usual spectral norm. A second one is that it is not recursive and it cannot be easily updated.

\subsection{A recursive and fast estimate}

We now give a recursive version of the algorithm defined by (\ref{algopel}) to estimate the asymptotic variance in separable Hilbert spaces, before establishing its rates of convergence (almost sure and in quadratic mean). This algorithm is defined by
\begin{equation}
\label{defisigman} \Sigma_{n} := \frac{1-\delta}{n^{1-\delta}} \sum_{k=1}^{n} \frac{1}{k^{\delta +s + \mu}}\exp \left( - \frac{k^{1-s}}{1-s}\right) \left( \sum_{j=1}^{k} j^{\mu /2}e^{\frac{j^{1-s}}{2(1-s)}}\left( m_{j} - \overline{m}_{j} \right) \right) \otimes \left( \sum_{j=1}^{k} j^{\mu /2} e^{\frac{j^{1-s}}{2(1-s)}}\left( m_{j} - \overline{m}_{j} \right) \right) , 
\end{equation}
with 
\begin{equation}
\label{condpas}(1 + \alpha)/2 < s < 1, \quad \quad \mu \geq 0, \quad \quad  \text{and} \quad \quad s/2 < \delta < (1 + s)/2. 
\end{equation}
The difference with previous algorithm is the replacement of $\overline{m}_{n}$ by $\overline{m}_{j}$, which enables the estimates to be written recursively for all $n \geq 1$ as
\begin{align*}
V_{n+1} & = V_{n} + (n+1)^{\mu /2}\exp \left( \frac{(n+1)^{1-s}}{2(1-s)}\right)\left( m_{n+1} - \overline{m}_{n+1}\right)  ,\\
\Sigma_{n+1} & = \left( \frac{n}{n+1}\right)^{1-\delta}\Sigma_{n} + \frac{1-\delta}{(n+1)^{\delta +s + \mu}}\exp \left( - \frac{(n+1)^{1-s}}{1-s}\right) V_{n+1} \otimes V_{n+1} ,
\end{align*} 
with $V_{1} = \Sigma_{1} = 0$. Then, contrary to previous algorithms, this one does not need to store all the estimations into memory and can be easily updated. Finally, the following theorem ensures that it is quite fast.

\begin{theo}\label{vitesselpps}
Suppose assumptions \textbf{(A1)} to \textbf{(A5a')} and \textbf{(A6b)} hold. Then, the sequence $\left(\Sigma_{n} \right)$ defined by (\ref{defisigman}) verifies for all positive constant $\gamma$, 
\[
\left\| \Sigma _{n} - \Sigma \right\|_{F}^{2} = o \left( \frac{(\ln n)^{\gamma}}{n^{1-s}}\right) \quad a.s .
\]
Moreover, suppose \textbf{(A5b)} holds too, there is a positive constant $C$ such that for all $n \geq 1$,
\[
\mathbb{E}\left[ \left\| \Sigma_{n} - \Sigma \right\|_{F}^{2} \right] \leq \frac{C}{n^{1-s}}
\]
\end{theo}
\noindent The proof is given in Section \ref{sectionproof}. 

\begin{cor}
Suppose assumptions \textbf{(A1)} to \textbf{(A5a')} and \textbf{(A6b)} hold. Then, for all positive constant $\gamma$, 
\[
\left\| \tilde{\Sigma}_{n} - \Sigma \right\|_{F}^{2} = o \left( \frac{(\ln n)^{\gamma}}{n^{1-s}}\right) \quad a.s .
\]
Moreover, suppose \textbf{(A5b)} holds too, there is a positive constant $C$ such that for all $n \geq 1$,
\[
\mathbb{E}\left[ \left\| \tilde{\Sigma}_{n} - \Sigma \right\|_{F}^{2} \right] \leq \frac{C}{n^{1-s}}
\]
\end{cor}
\begin{rmq}
The constant $C$ in Theorem \ref{vitesselpps} depends on the constants introduced in assumptions, on the initialization of the stochastic gradient descent, and on $\alpha,\delta , \mu , s,c_\gamma $. 
\end{rmq}
\begin{rmq}
Estimating recursively the asymptotic variance coupled with Theorem  \ref{theotlc} can be useful to build online asymptotic confidence balls. Moreover, in the recent literature, non asymptotic convergence rates are often given under the form
\[
\mathbb{E}\left[ \left\| \overline{m}_{n} - m \right\|^{2} \right] \leq \frac{\left\| \Sigma \right\|_{F}}{n} + R_{n},
\]
where $R_{n}$ is a rest term. Then, using the recursive variance estimates could enable to have, in practice, a precise bound of the quadratic mean error, and in the short term, it could allow to get precise non asymptotic confidence balls.
\end{rmq}
\begin{rmq}
In order to get a faster algorithm (in term of computational time), one can consider a parallelized version of previous estimates. This consists in splitting the sample into $p$ parts, and to run the algorithm on each subsample to get $p$ estimates $\Sigma_{n/p,i}$, before taking the mean of these $p$ last ones.
\end{rmq}

\section{Applications}\label{sectionapplication}
\subsection{Application to the logistic regression}
Let $d$ be a positive integer, and let $Y \in \left\lbrace -1,1 \right\rbrace$ and $X \in \mathbb{R}^{d}$ be random variables. In order to get the parameter $m^{l} \in \mathbb{R}^{d}$ of the logistic regression, the aim is to minimize the functional $G_{l}$ defined for all $h \in \mathbb{R}^{d}$ by
\begin{equation}
G_{l}(h) := \mathbb{E}\left[ \log \left( 1+ \exp \left( -Y \left\langle X,h \right\rangle \right) \right) \right] .
\end{equation}
Under usual assumptions (see \cite{bach2014adaptivity} among others), the functional $G_{l}$ is locally strongly convex and twice Fréchet differentiable with for all $h \in \mathbb{R}^{d}$,
\begin{align*}
& \nabla G_{l}(h) = - \mathbb{E}\left[ \frac{\exp \left( -Y \left\langle X,h \right\rangle \right)}{1+\exp \left( -Y \left\langle X,h \right\rangle \right)}YX \right] ,
& \nabla^{2}G_{l}(h) = \mathbb{E}\left[ \frac{\exp \left( -Y \left\langle X,h \right\rangle \right)}{\left( 1+\exp \left( -Y \left\langle X,h \right\rangle \right)\right)^{2}}        X\otimes X \right] .
\end{align*}
Then, the parameters of the logistic regression and the asymptotic variance can be estimated simultaneously as:
\begin{align*}
& m_{n+1}^{l} = m_{n}^{l} + \gamma_{n} \frac{\exp \left( -Y_{n+1} \left\langle X_{n+1},m_{n}^{l} \right\rangle \right)}{1+\exp \left( -Y_{n+1} \left\langle X_{n+1},m_{n}^{l} \right\rangle \right)}Y_{n+1}X_{n+1} , \\
& \overline{m}_{n+1}^{l} = \overline{m}_{n}^{l} + \frac{1}{n+1}\left( m_{n+1}^{l} - \overline{m}_{n}^{l} \right) , \\
& V_{n+1}^{l} = V_{n}^{l} +  (n+1)^{\mu /2}\exp \left( \frac{(n+1)^{1-s}}{2(1-s)}\right)\left( m_{n+1}^{l} - \overline{m}_{n+1}^{l} \right) , \\
& \Sigma_{n+1}^{l} = \left( \frac{n}{n+1}\right)^{1-\delta}\Sigma_{n}^{l} + \frac{1-\delta}{(n+1)^{\delta +s + \mu}}\exp \left( - \frac{(n+1)^{1-s}}{1-s} \right) V_{n+1}^{l} \otimes V_{n+1}^{l} .
\end{align*}

\subsection{Application to the geometric median and geometric quantiles}
Let $H$ be a separable Hilbert space and let $X$ be a random variable taking values in $H$. Let $v~ \in ~H$ such that $\left\| v \right\| < 1$, the geometric quantile $m^{v}$ corresponding to the direction $v$ (see \cite{Cha96}) is defined by
\begin{equation}
m^{v} := \arg \min_{h \in H} \mathbb{E}\left[ \left\| X-h \right\| - \left\| X \right\| \right] - \left\langle h, v \right\rangle ,
\end{equation}
and in a particular case, the geometric median $m$ (see \cite{Hal48}) corresponds to the case where $v=0$. Under usual assumptions (see \cite{Kem87} and \cite{HC} among others), the functional $G_{v}$ is locally strongly convex and twice Fréchet-differentiable with for all $h \in H$,
\begin{align*}
& \nabla G^{v}(h) = - \mathbb{E}\left[ \frac{X-h}{\left\| X-h \right\|} +v \right] , & \nabla^{2}G^{v}(h) = \mathbb{E}\left[ \frac{1}{\left\| X-h \right\|} \left( I_{H} - \frac{\left( X-h \right) \otimes \left( X-h \right)}{\left\| X-h \right\|^{2}}\right)\right].
\end{align*}
Then, it is possible to estimate simultaneously and recursively the geometric quantile $m^{v}$ as well as the asymptotic variance of the averaged estimator as follows:
\begin{align*}
& m_{n+1}^{v} = m_{n}^{v} + \gamma_{n} \left( \frac{X_{n+1}-m_{n}^{v}}{\left\| X_{n+1} - m_{n}^{v} \right\|} + v \right) , \\
& \overline{m}_{n+1}^{v} = \overline{m}_{n}^{v} + \frac{1}{n+1}\left( m_{n+1}^{v} - \overline{m}_{n}^{v} \right) , \\
& V_{n+1}^{v} = V_{n}^{v} +  (n+1)^{\mu /2}\exp \left( \frac{(n+1)^{1-s}}{2(1-s)}\right)\left( m_{n+1}^{v} - \overline{m}_{n+1}^{v} \right) , \\
& \Sigma_{n+1}^{v} = \left( \frac{n}{n+1}\right)^{1-\delta}\Sigma_{n} + \frac{1-\delta}{(n+1)^{\delta +s + \mu}}\exp \left( - \frac{(n+1)^{1-s}}{1-s} \right) V_{n+1}^{v} \otimes V_{n+1}^{v} .
\end{align*}
Note that under usual assumptions, the asymptotic variance obtained is the same as the one obtained with non-recursive estimates  \citep{MR2238141,Ger08} in the special case of the geometric median. 

\subsection{A short simulation study}

We focus here on the estimation of the geometric median. We consider from now that $X$ is a random variable taking values in $\mathbb{R}^{d}$, with $d \geq 3$, and following a uniform law on the unit sphere $\mathcal{S}^{d}$. Then, the geometric median $m$ is equal to $0$ and the Hessian of the functional $G_{0}$ at $m$ verifies
\[
\Gamma_{m} = \mathbb{E}\left[ \frac{1}{\| X \|}\left( I_{d} - \frac{X}{\| X \|} \otimes \frac{X}{\| X \|} \right) \right] = I_{d} -  \mathbb{E} \left[ X \otimes X \right] = \frac{d-1}{d}I_{d}.
\]
Note that assumptions \textbf{(A1)} and \textbf{(A6b)} are then verified (see Section 3 in \cite{godichon2016}, Lemma A.1 in \cite{GBP2016} and the Appendix to be convinced).
Finally, the asymptotic variance of the stochastic gradient estimate and of its averaged version verify
\begin{align*}
& \Sigma_{RM}  = \int_{0}^{\infty} e^{-s\Gamma_{m}}\mathbb{E}\left[ \frac{X}{\| X \|}\otimes\frac{X}{\| X \|} \right] e^{-s\Gamma_{m}} ds = \frac{1}{2(d-1)}I_{d}, \\
& \Sigma  = \Gamma_{m}^{-1} \mathbb{E}\left[ \frac{X}{\| X \|}\otimes\frac{X}{\| X \|} \right] \Gamma_{m}^{-1}  = \frac{d}{(d-1)^2} I_{d}.
\end{align*}
First, let us consider a stepsequence $\gamma_{n}=n^{-2/3}$ and let us study the quality of the Gaussian approximation of $Q_{n}, Q_{n}'$, where 
\[
Q_{n}:= \sqrt{2(d-1)}n^{1/3}\left( m_{n} - m \right) ,\quad \quad \text{and} \quad \quad Q_{n}':= \sqrt{n}\frac{d-1}{\sqrt{d}}\left( \overline{m}_{n} - m \right).
\]
Figure \ref{fig0rm} (respectively Figure \ref{fig0})  seems to confirm Theorem \ref{tlcgrad} (respectively Theorem \ref{theotlc}) since we can see that the estimated density of a component of $Q_{n}$ (respectively $Q_{n}'$) is close to the density of $\mathcal{N}\left( 0,1 \right)$, and so, even for small sample sizes ($n=200$), which is also confirmed by a Kolmogorov-Smirnov test.

\begin{figure}[H]
\begin{center}
\includegraphics[scale=0.4]{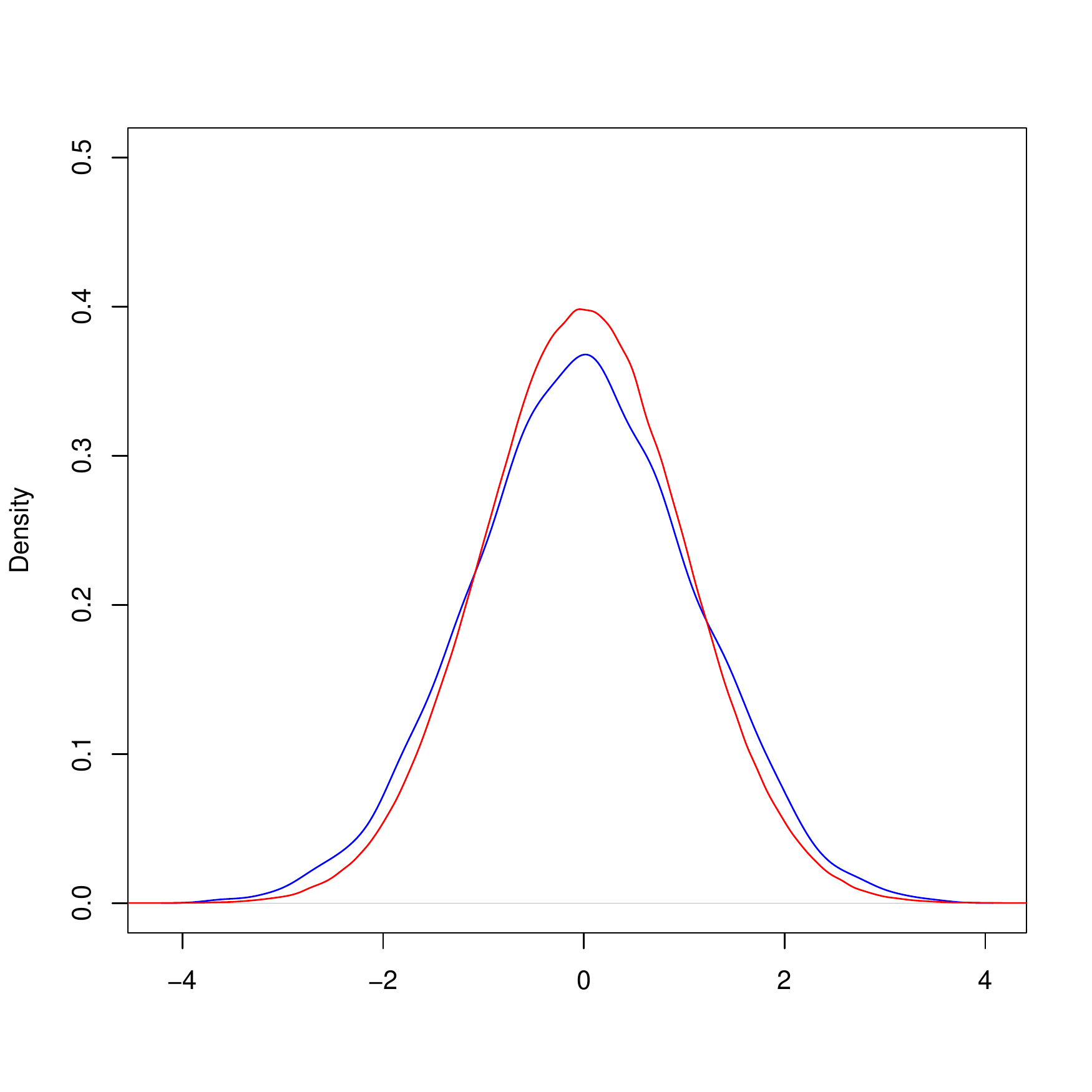}
\includegraphics[scale=0.4]{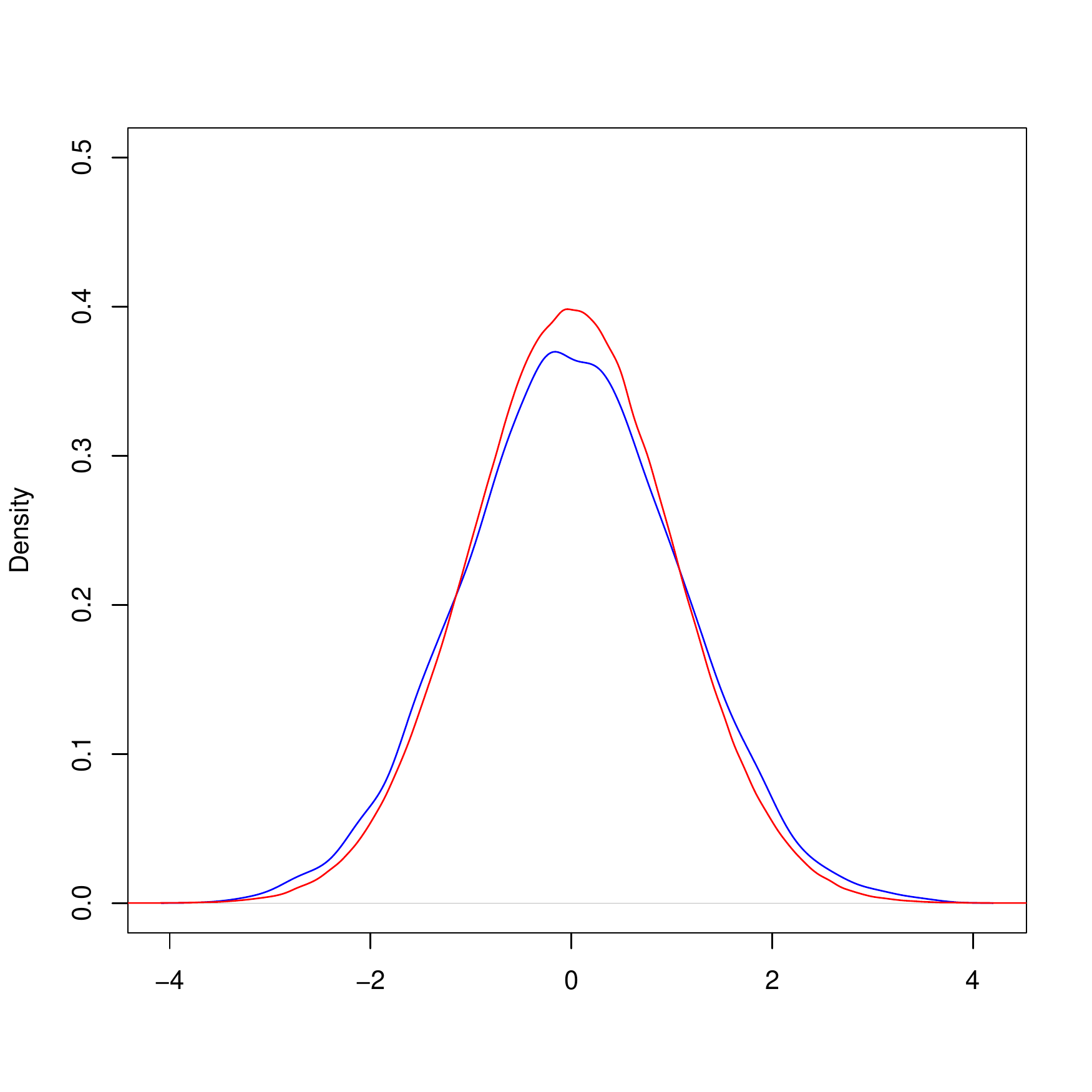}
\caption{Estimated density of a component of $Q_{n}$ (in blue) compared to the standard gaussian density (in red), with $n=200$ (on the left) and $n=5000$ (on the right).}\label{fig0rm}
\end{center}
\end{figure}

\begin{figure}[H]
\begin{center}
\includegraphics[scale=0.4]{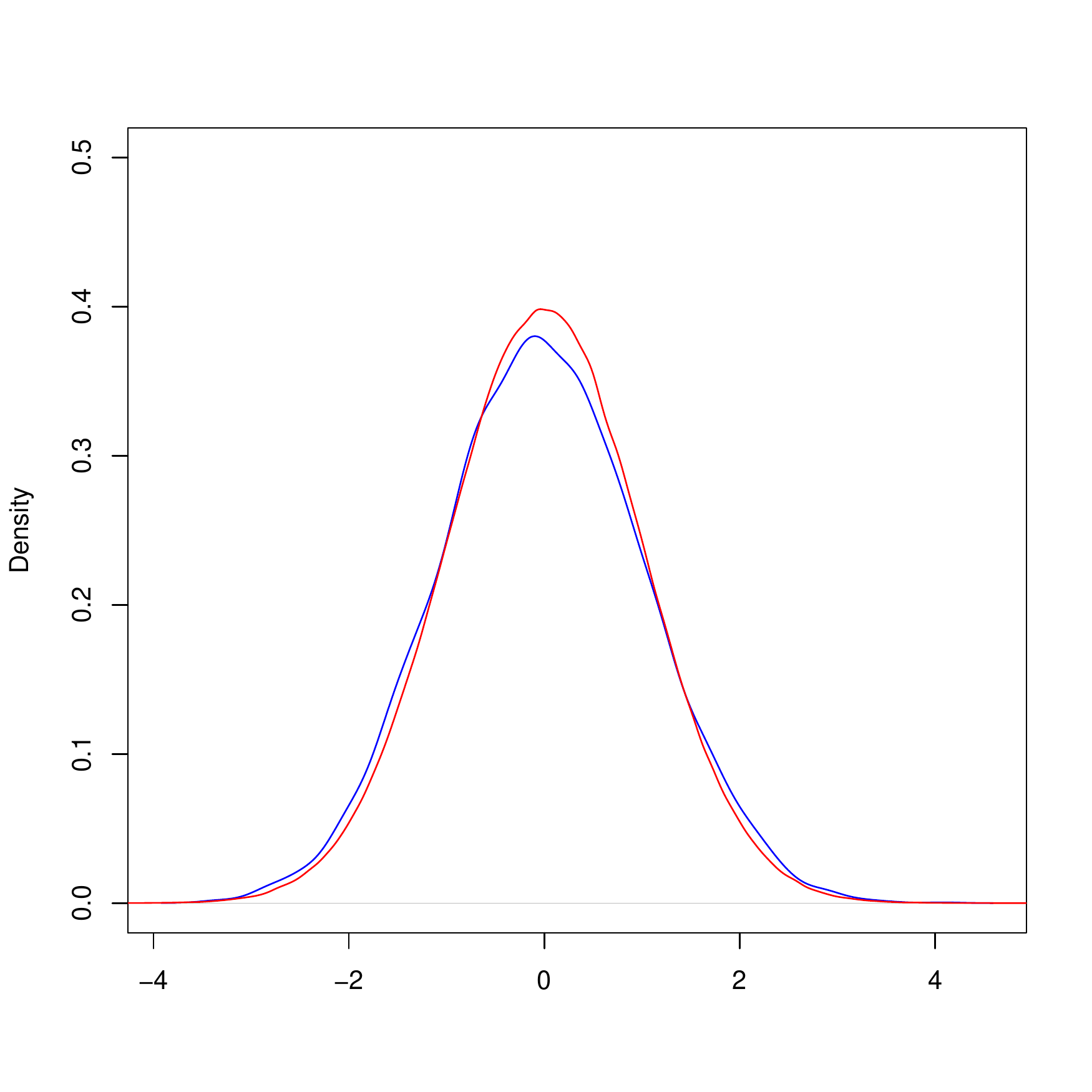}
\includegraphics[scale=0.4]{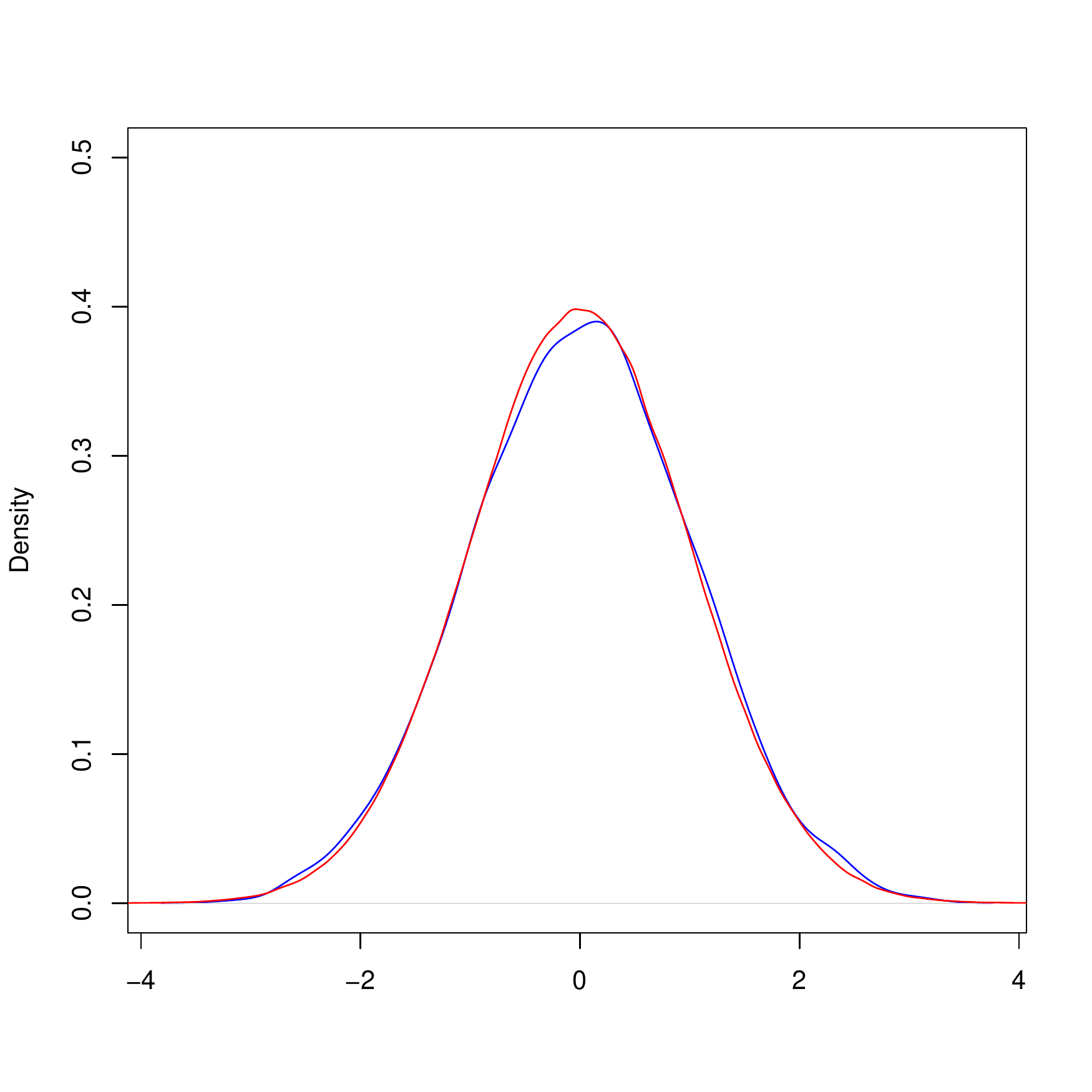}
\caption{Estimated density of a component of $Q_{n}'$ (in blue) compared to the standard gaussian density (in red), with $n=200$ (on the left) and $n=5000$ (on the right).}\label{fig0}
\end{center}
\end{figure}

In Figure \ref{fig1}, we consider the evolution of the quadratic mean error, with respect to the Frobenius norm, of the estimates $\left( \Sigma_{n} \right)$ of $\Sigma$ defined by (\ref{defisigman}), with regard to the sample size. For this, we generate $100$ samples, and use the parallelized version of the algorithms. Figure~\ref{fig1} tends to confirm that for small dimensional spaces ($d=10$), the estimates of the asymptotic variance converge quite quickly and that it is still the case for moderate dimensional spaces ($d=5000$).

\begin{figure}[H]
\begin{center}
\includegraphics[scale=0.4]{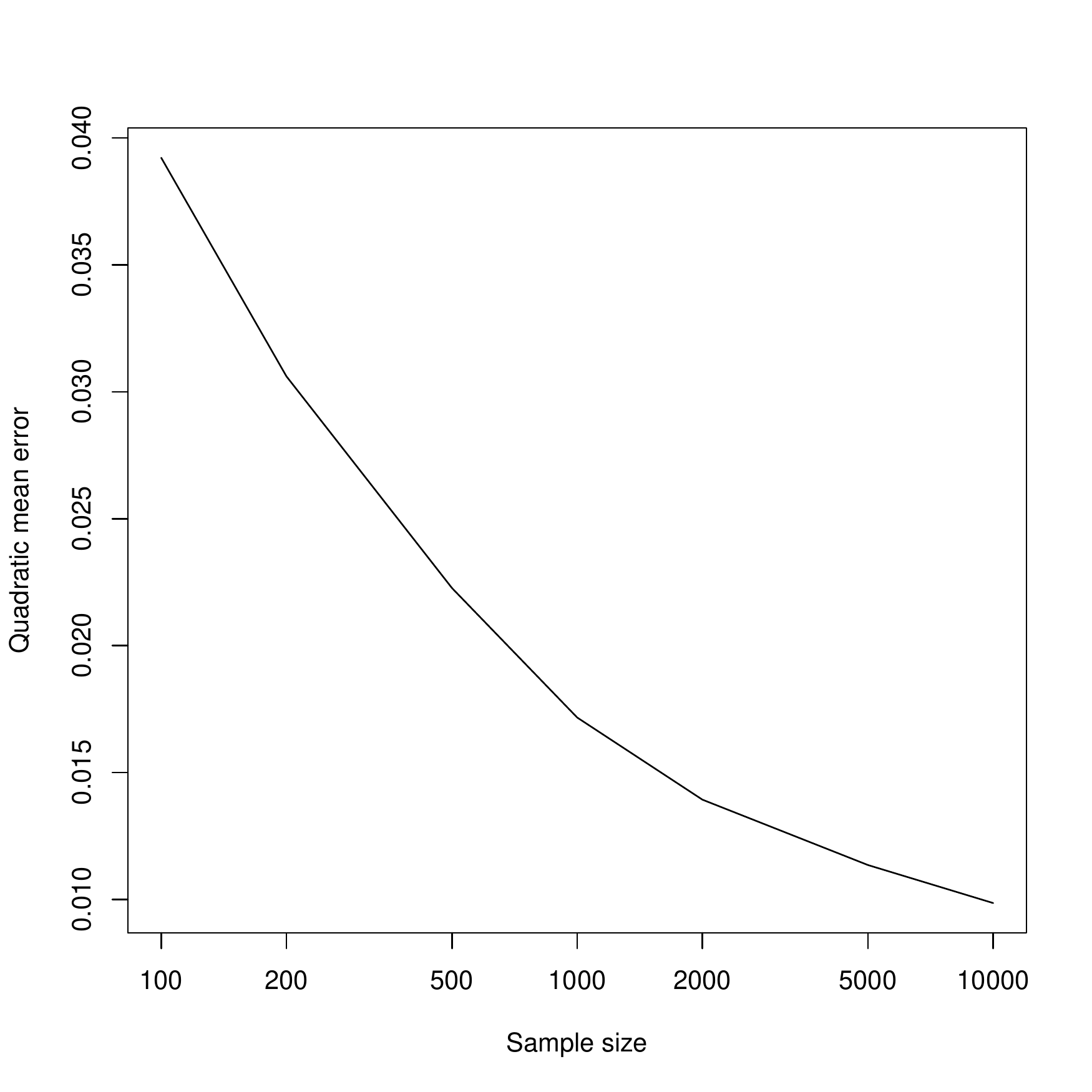}
\includegraphics[scale=0.4]{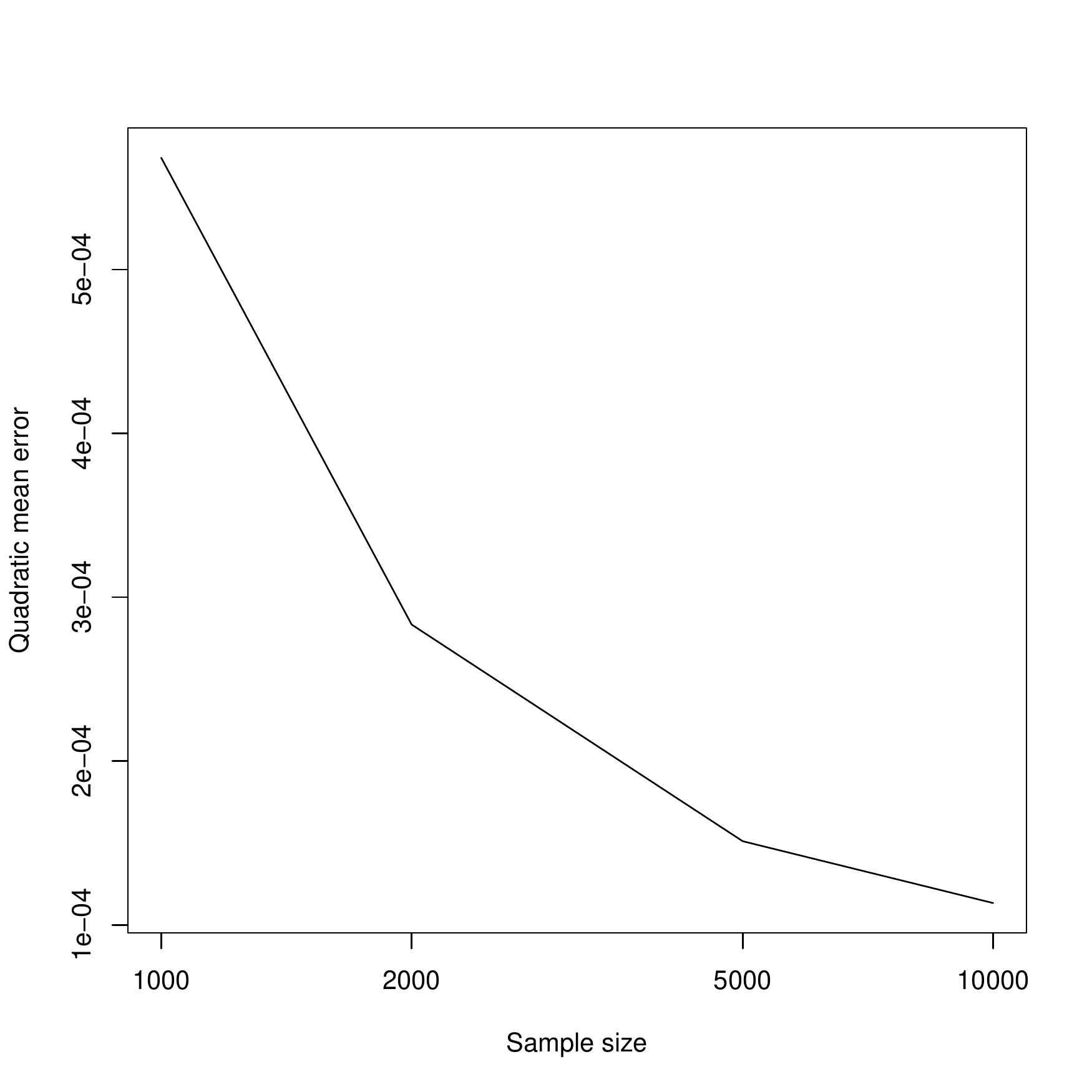}
\caption{Evolution of the quadratic mean error of the estimation of the asymptotic variance $\Sigma$ with respect to the Frobenius norm for $d=10$ (on the left) and $d=5000$ (on the right).}\label{fig1}
\end{center}
\end{figure}

\section{Proofs}\label{sectionproof}
\subsection{Some decompositions of the algorithms}
In order to simplify the proofs, let us now give some decompositions of the algorithms.

\subsubsection{The Robbins-Monro algorithm} Let us recall that the stochastic gradient algorithm can be written as
\[
m_{n+1} - m = m_{n} - m - \gamma_{n} \Phi \left( m_{n} \right) + \gamma_{n} \xi_{n+1} .
\]
Linearizing the gradient, it comes
\begin{equation}
\label{decdelta} m_{n+1} - m = \left( I_{H} - \gamma_{n} \Gamma_{m} \right)\left( m_{n} - m \right) + \gamma_{n} \xi_{n+1} - \gamma_{n} \delta_{n}, 
\end{equation}
where $\delta_{n} := \Gamma_{m} \left( m_{n} - m \right) - \Phi \left( m_{n} \right)$ is the remainder term in the Taylor's expansion of the gradient. 
Thanks to previous decomposition and with the help of an induction (see \cite{dufloalgo} or \cite{Duf97} for instance), one can check that for all $n \geq 1$,
\begin{equation}
\label{decbeta} m_{n} - m = \beta_{n-1} \left( m_{1} - m \right)  - \beta_{n-1} \sum_{k=1}^{n-1}\gamma_{k} \beta_{k}^{-1}\delta_{k} + \beta_{n-1} \sum_{k=1}^{n-1}\gamma_{k} \beta_{k}^{-1}\xi_{k+1} ,
\end{equation}
with $\beta_{n} := \prod_{k=1}^{n} \left( I_{H} - \gamma_{k} \Gamma_{m} \right)$ for all $n \geq 1$ and $\beta_{0} := I_{H}$.
Finally, the asymptotic variance can be seen as the almost sure limit of the sequence of random variables $\left( \Gamma_{m}^{-1}\xi_{n} \otimes \Gamma_{m}^{-1}\xi_{n} \right)_{n}$ (see the proof of Theorem \ref{theotlc}). Then, in order to prove the convergence of the estimates, we need to exhibit this sequence. In this aim, one can rewrite equation (\ref{decdelta}) as 
\begin{equation}
\label{decdeltabis}m_{n} - m = \frac{T_{n}}{\gamma_{n}}- \frac{T_{n+1}}{\gamma_{n}} + \Xi_{n+1} - \Delta_{n},
\end{equation}
with
\[ T_{n} := \Gamma_{m}^{-1} \left( m_{n} - m \right) , \quad \quad  \Xi_{n+1} := \Gamma_{m}^{-1} \left( \xi_{n+1} \right) , \quad \quad  \Delta_{n} := \Gamma_{m}^{-1} \left( \delta_{n} \right) .
\]

\subsubsection{The averaged algorithm}
Summing equalities (\ref{decdeltabis}) and dividing by $n$, we obtain the following decomposition of the averaged estimator
\begin{equation}
\label{decmoybis} \overline{m}_{n} - m = \frac{1}{n} \sum_{k=1}^{n}\left( \frac{T_{k}}{\gamma_{k}} - \frac{T_{k+1}}{\gamma_{k}}\right) - \frac{1}{n}\sum_{k=1}^{n} \Delta_{k} + \frac{1}{n}\sum_{k=1}^{n} \Xi_{k+1} . 
\end{equation}
Finally, by linearity and applying an Abel's transform to the first term on the right-hand side of previous equality (see \cite{delyon1992stochastic} or \cite{delyon1993accelerated} for instance), 
\begin{align}
\notag \Gamma_{m} \left( \overline{m}_{n} - m \right) & = \frac{m_{1} - m}{n\gamma_{1}} - \frac{m_{n+1}-m}{n\gamma_{n}} + \frac{1}{n}\sum_{k=2}^{n} \left( \frac{1}{\gamma_{k}} - \frac{1}{\gamma_{k-1}} \right) \left( m_{k} - m \right) - \frac{1}{n}\sum_{k=1}^{n} \delta_{k} \\
\label{decmoy} & + \frac{1}{n}\sum_{k=1}^{n}\xi_{k+1}.
\end{align}

\subsubsection{The recursive estimator of the asymptotic variance}
In order to simplify the proof of Theorem \ref{vitesselpps}, we will introduce a new estimator of the variance. In this aim, let us now introduce the sequences $\left( a_{n} \right)_{n \geq 1}$ and $\left( b_{n} \right)_{n \geq 1}$ defined for all $n \geq 1$ by $a_{n} := \exp \left( \frac{n^{1-s}}{2(1-s)} \right)$ and $b_{n}:= \sum_{k=1}^{n}a_{k}^{2}$. Then, thanks to decomposition (\ref{decdeltabis}), let
\begin{align}
\notag\overline{T}_{n} & := \frac{1}{\sqrt{b_{n}}}\sum_{k=1}^{n} a_{k} \left( m_{k} - m \right) \\
\notag & = \frac{1}{\sqrt{b_{n}}}\left(  \sum_{k=1}^{n} \frac{a_{k}}{\gamma_{k}} \left( T_{k} - T_{k+1} \right) + \sum_{k=1}^{n} a_{k} \Delta_{k} + \sum_{k=1}^{n} a_{k}\Xi_{k+1} \right) \\
\label{defiai}& =: \frac{1}{\sqrt{b_{n}}}\left( A_{1,n} + A_{2,n} + M_{n+1} \right) . 
\end{align}
In order to simplify several proofs, we now give $L^{p}$ upper bounds of the terms on the right-hand side of previous equality.
\begin{lem}\label{lemtech}
Suppose assumptions \textbf{(A1)} to \textbf{(A5b)} hold. Then, for all positive integer $p$,
\begin{align*}
& \mathbb{E}\left[ \left\| \sum_{k=1}^{n} \frac{a_{k}}{\gamma_{k}} \left( T_{k} - T_{k+1} \right) \right\|^{2p} \right] =  O \left( \exp \left( \frac{pn^{1-s}}{1-s} \right) n^{p\alpha } \right) , \\
& \mathbb{E}\left[ \left\| \sum_{k=1}^{n} a_{k} \Delta_{k} \right\|^{2p} \right] = O \left( \exp \left( \frac{pn^{1-s}}{1-s} \right) n^{p(s-\alpha } \right) , \\
& \mathbb{E}\left[ \left\| \sum_{k=1}^{n} a_{k}\Xi_{k+1} \right\|^{2p}\right] = O \left( \exp \left( \frac{pn^{1-s}}{1-s} \right) n^{p s } \right)
\end{align*}
\end{lem}
The proof of this lemma as well as an analogous lemma which gives the asymptotic almost sure behavior of these terms are given in Appendix. We can now introduce the following estimator
\begin{equation}
\overline{\Sigma}_{n} = \frac{1}{\sum_{k=1}^{n}k^{- \delta}}\sum_{k=1}^{n}\frac{1}{k^{\delta}}T_{k}\otimes T_{k} ,
\end{equation}
and one can decompose $\Sigma_{n}$ as follows:
\begin{align*}
\Sigma_{n} - \Sigma & = \Sigma_{n} - \frac{1-\delta}{n^{1-\delta}} \sum_{k=1}^{n} \frac{1}{k^{\delta +s}}\exp \left( - \frac{k^{1-s}}{1-s}\right) \left( \sum_{j=1}^{k} e^{\frac{j^{1-s}}{2(1-s)}}\left( m_{j} - m \right) \right) \otimes \left( \sum_{j=1}^{k} e^{\frac{j^{1-s}}{2(1-s)}}\left( m_{j} - m \right) \right) \\
&  + \frac{1-\delta}{n^{1-\delta}} \sum_{k=1}^{n} \frac{1}{k^{\delta +s}}\exp \left( - \frac{k^{1-s}}{1-s}\right) \left( \sum_{j=1}^{k} e^{\frac{j^{1-s}}{2(1-s)}}\left( m_{j} - m \right) \right) \otimes \left( \sum_{j=1}^{k} e^{\frac{j^{1-s}}{2(1-s)}}\left( m_{j} - m \right) \right) - \overline{\Sigma}_{n} \\
& + \overline{\Sigma}_{n} - \Sigma .
\end{align*}

\subsection{Proof of Theorem \ref{theotlc}}
\begin{proof}[Proof of Theorem \ref{theotlc}]
Let us recall that the averaged algorithm can be written as
\begin{align*}
 \Gamma_{m} \left( \overline{m}_{n} - m \right) & = \frac{m_{1} - m}{n\gamma_{1}} - \frac{m_{n+1}-m}{n\gamma_{n}} + \frac{1}{n}\sum_{k=2}^{n} \left( \frac{1}{\gamma_{k}} - \frac{1}{\gamma_{k-1}} \right) \left( m_{k} - m \right) - \frac{1}{n}\sum_{k=1}^{n} \delta_{k} \\
 & + \frac{1}{n}\sum_{k=1}^{n}\xi_{k+1}.
\end{align*}
It is proven in \cite{godichon2016} that
\begin{align*}
\frac{\left\| m_{1}-m\right\| }{\sqrt{n}\gamma_{1}} & = o \left( 1 \right) \quad a.s , \\
\frac{\left\| m_{n+1}-m\right\|}{\sqrt{n}\gamma_{n}} & = o ( 1 ) \quad a.s , \\
\frac{1}{\sqrt{n}}\left\| \sum_{k=2}^{n} \left( \frac{1}{\gamma_{k}} - \frac{1}{\gamma_{k-1}}\right)\left( m_{k}-m \right)\right\| & = o ( 1 ) \quad a.s , \\
\frac{1}{\sqrt{n}}\left\| \sum_{k=1}^{n}\delta_{k}\right\| & = o (1) \quad a.s.
\end{align*}
In order get the asymptotic normality of the martingale term $\left( \frac{1}{n}\sum_{k=1}^{n}\xi_{k+1} \right)$, let us check that assumptions of Theorem 5.1 in \cite{Jak88} are fulfilled, i.e let $\left( e_{i} \right)_{i \in I}$ be an orthonormal basis of $H$ and $\psi_{i,j}:= \left\langle \Sigma ' e_{i},e_{j} \right\rangle$ for all $i,j \in I$, we have to verify
\begin{equation}
\label{cond1}\forall \eta > 0, \quad \lim_{n \to \infty} \mathbb{P}\left( \sup_{1 \leq k \leq n} \frac{1}{\sqrt{n}}\left\| \xi_{k+1} \right\| > \eta \right) = 0 ,
\end{equation}
\begin{equation}
\label{cond2} \lim_{n \to \infty} \frac{1}{n}\sum_{k=1}^{n} \left\langle \xi_{k+1} , e_{i} \right\rangle \left\langle \xi_{k+1} , e_{j} \right\rangle = \psi_{i,j} \quad a.s, \quad \forall i,j \in I ,
\end{equation}
\begin{equation}
\label{cond3}\forall \epsilon > 0, \quad \lim_{N \to \infty} \limsup_{n \to \infty} \mathbb{P}\left( \frac{1}{n}\sum_{k=1}^{n} \sum_{j=N}^{\infty}\left\langle \xi_{k+1},e_{j} \right\rangle^{2} > \epsilon \right) = 0.
\end{equation}

\textbf{Proof of (\ref{cond1})} Let $\eta > 0$, applying Markov's inequality,
\begin{align*}
\mathbb{P}\left( \sup_{1 \leq k \leq n}\frac{1}{\sqrt{n}}\left\| \xi_{k+1} \right\| > \eta \right) & \leq \sum_{k=1}^{n} \mathbb{P}\left( \frac{1}{\sqrt{n}}\left\| \xi_{k+1} \right\| > \eta \right) \\
& \leq \frac{1}{n^{2}\eta^{4}}\sum_{k=1}^{n} \mathbb{E} \left[ \left\| \xi_{k+1} \right\|^{4} \right] .
\end{align*}
Then, applying Lemma \ref{lemmajxi}, there is a positive constant $C$ such that
\begin{align*}
\mathbb{P}\left( \sup_{1 \leq k \leq n}\frac{1}{\sqrt{n}}\left\| \xi_{k+1} \right\| > \eta \right) & \leq \frac{1}{n^{2}\eta^{4}}\sum_{k=1}^{n}C = \frac{C}{n\eta^{4}} .
\end{align*}

\medskip

\textbf{Proof of (\ref{cond2}).} First, note that
\[
\frac{1}{n}\sum_{k=1}^{n} \xi_{k+1} \otimes \xi_{k+1} = \frac{1}{n}\sum_{k=1}^{n} \mathbb{E}\left[ \xi_{k+1} \otimes \xi_{k+1} |\mathcal{F}_{k} \right] + \frac{1}{n}\sum_{k=1}^{n} \epsilon_{k+1},
\]
with $\epsilon_{k+1} := \xi_{k+1} \otimes \xi_{k+1} - \mathbb{E} \left[ \xi_{k+1} \otimes \xi_{k+1} |\mathcal{F}_{k} \right]$. Remark that $\left( \epsilon_{n} \right)$ is a sequence of martingale differences adapted to the filtration $\left( \mathcal{F}_{n} \right)$, and one can check that
\[
\lim_{n \to \infty} \frac{1}{n}\sum_{k=1}^{n} \epsilon_{k+1} = 0 \quad a.s.
\]
Let us now prove that the sequence of operators $\left( \mathbb{E}\left[ \xi_{k+1} \otimes \xi_{k+1} |\mathcal{F}_{k} \right] \right)$ converges almost surely to $\Sigma '$, with respect to the Frobenius norm. Note that
\begin{align*}
\left\| \mathbb{E}\left[ \xi_{k+1} \otimes \xi_{k+1} |\mathcal{F}_{k} \right] - \Sigma ' \right\| & = \left\| \mathbb{E}\left[ \nabla_{h} g \left( X_{k+1} , m_{k} \right) \otimes \nabla_{h} g \left( X_{k+1} , m_{k} \right)|\mathcal{F}_{k} \right] - \Sigma ' - \Phi \left( m_{k} \right) \otimes \Phi \left(m_{k} \right)\right\|_{F} \\
& \leq \left\| \mathbb{E}\left[ \nabla_{h} g \left( X_{k+1} , m_{k} \right) \otimes \nabla_{h} g \left( X_{k+1} , m_{k} \right)|\mathcal{F}_{k} \right] - \Sigma ' \right\|_{F} + \left\|  \Phi \left( m_{k} \right) \otimes \Phi \left(m_{k} \right)\right\|_{F}.
\end{align*}
Then, thanks to assumption \textbf{(A6a)}, since $\left\| \Phi (m_{k} ) \right\| \leq C \left\| m_{k} - m \right\|$ and since $\left( m_{k} \right)$ converges to $m$ almost surely (see \cite{godichon2016}),
\begin{align*}
 & \lim_{k \to \infty}\left\| \mathbb{E}\left[ \nabla_{h} g \left( X_{k+1} , m_{k} \right) \otimes \nabla_{h} g \left( X_{k+1} , m_{k} \right)|\mathcal{F}_{k} \right] - \Sigma ' \right\|_{F} = 0 \quad a.s, \\
& \lim_{k \to \infty} \left\|  \Phi \left( m_{k} \right) \otimes \Phi \left(m_{k} \right)\right\|_{F} = \lim_{n \to \infty} \left\|  \Phi (m_{k} ) \right\|^{2} = 0 \quad a.s.
\end{align*}
In a particular case, for all $i,j \in I$,
\[
\lim_{k \to \infty} \left\langle \mathbb{E}\left[ \xi_{k+1} \otimes \xi_{k+1} |\mathcal{F}_{k} \right] (e_{i}) , e_{j} \right\rangle = \psi_{i,j} := \left\langle \Sigma ' (e_{i} ) , e_{j} \right\rangle \quad a.s .  
\] 
Thus, applying Toeplitz's lemma,
\[
\lim_{n \to \infty} \frac{1}{n}\sum_{k=1}^{n}\left\langle \mathbb{E}\left[ \xi_{k+1} \otimes \xi_{k+1} |\mathcal{F}_{k} \right] (e_{i}) , e_{j} \right\rangle = \psi_{i,j} \quad a.s.
\]
Finally, for all $ i,j \in I$,
\begin{align*}
\lim_{n \to \infty} \frac{1}{n}\sum_{k=1}^{n} \left\langle \xi_{k+1} , e_{i} \right\rangle \left\langle \xi_{k+1} , e_{j} \right\rangle & = \lim_{n \to \infty} \frac{1}{n}\sum_{k=1}^{n}\left\langle \xi_{k+1}\otimes \xi_{k+1} (e_{i}),e_{j} \right\rangle \\
& = \psi_{i,j} \quad a.s .
\end{align*}

\textbf{Proof of (\ref{cond3}).} Let $\epsilon > 0$, applying Markov's inequality,
\begin{align*}
\mathbb{P}\left( \frac{1}{n}\sum_{k=1}^{n} \sum_{j=N}^{\infty} \left\langle \xi_{k+1}, e_{j} \right\rangle > \epsilon \right) & \leq \frac{1}{n\epsilon^{2}} \sum_{k=1}^{n} \sum_{j=N}^{\infty} \mathbb{E}\left[ \left\langle \xi_{k+1} , e_{j} \right\rangle^{2} \right] \\
& =\frac{1}{n\epsilon^{2}} \sum_{k=1}^{n} \sum_{j=N}^{\infty} \mathbb{E}\left[ \mathbb{E}\left[ \left\langle \xi_{k+1} , e_{j} \right\rangle^{2}|\mathcal{F}_{k} \right] \right] .
\end{align*}
Since for all $j \in I$, $\left\langle \xi_{k+1} , e_{j} \right\rangle^{2} = \left\langle \xi_{k+1} \otimes \xi_{k+1} (e_{j} , e_{j} \right\rangle$, and by linearity 
\begin{align*}
\mathbb{P}\left( \frac{1}{n}\sum_{k=1}^{n} \sum_{j=N}^{\infty} \left\langle \xi_{k+1}, e_{j} \right\rangle > \epsilon \right) & \leq  \frac{1}{\epsilon^{2}}  \sum_{j=N}^{\infty} \frac{1}{n}\sum_{k=1}^{n} \mathbb{E}\left[  \mathbb{E}\left[ \left\langle \xi_{k+1} \otimes \xi_{k+1} (e_{j}) , e_{j} \right\rangle |\mathcal{F}_{k} \right] \right] \\
& = \frac{1}{\epsilon^{2}}  \sum_{j=N}^{\infty} \frac{1}{n}\sum_{k=1}^{n} \mathbb{E}\left[   \left\langle \mathbb{E}\left[ \xi_{k+1} \otimes \xi_{k+1} |\mathcal{F}_{k} \right] (e_{j}) , e_{j} \right\rangle  \right] .
\end{align*} 
Since $\mathbb{E}\left[ \xi_{k+1} \otimes \xi_{k+1} |\mathcal{F}_{k} \right]$ converges almost surely to $\Sigma '$ and by dominated convergence, 
\[
\limsup_{n}\mathbb{P}\left( \frac{1}{n}\sum_{k=1}^{n} \sum_{j=N}^{\infty} \left\langle \xi_{k+1}, e_{j} \right\rangle > \epsilon \right) \leq \frac{1}{\epsilon} \sum_{j=N}^{\infty} \left\langle \Sigma ' (e_{j}) , e_{j} \right\rangle .
\]
Moreover, since $\Sigma ' = \mathbb{E}\left[ \nabla_{h} g \left( X , m \right) \otimes \nabla_{h} g \left( X , m \right) \right]$, thanks to assumption \textbf{(A5a)},
\begin{align*}
\sum_{j=1}^{\infty} \left\langle \Sigma ' (e_{j}) , e_{j} \right\rangle & = \left\| \mathbb{E}\left[ \nabla_{h} g \left( X , m \right) \otimes \nabla_{h} g \left( X , m \right) \right] \right\|_{F}  \leq \mathbb{E}\left[ \left\| \nabla_{h} g \left( X , m \right) \right\|^{2} \right]  \leq L_{1} .
\end{align*}
Thus, since for all $j \in I$, $\left\langle \Sigma ' (e_{j}) , e_{j} \right\rangle \geq 0$,
\[
\lim_{N \to \infty} \sum_{j=N}^{\infty} \left\langle \Sigma ' (e_{j}) , e_{j} \right\rangle = 0,
\]
which concludes the proof.
\end{proof}

\subsection{Proof of Theorem \ref{vitesselpps}}
For the sake of simplicity, the proof is given for $mu = 0$ (the case where $\mu > 0$ is strictly analogous). Let us recall that equation (\ref{defisigman}) can be written  as
\begin{align}
\notag \Sigma_{n} - \Sigma & = \Sigma_{n} - \frac{1-\delta}{n^{1-\delta}} \sum_{k=1}^{n} \frac{1}{k^{\delta +s}}\exp \left( - \frac{k^{1-s}}{1-s}\right) \left( \sum_{j=1}^{k} e^{\frac{j^{1-s}}{2(1-s)}}\left( m_{j} - m \right) \right) \otimes \left( \sum_{j=1}^{k} e^{\frac{j^{1-s}}{2(1-s)}}\left( m_{j} - m \right) \right) \\
\notag &  + \frac{1-\delta}{n^{1-\delta}} \sum_{k=1}^{n} \frac{1}{k^{\delta +s}}\exp \left( - \frac{k^{1-s}}{1-s}\right) \left( \sum_{j=1}^{k} e^{\frac{j^{1-s}}{2(1-s)}}\left( m_{j} - m \right) \right) \otimes \left( \sum_{j=1}^{k} e^{\frac{j^{1-s}}{2(1-s)}}\left( m_{j} - m \right) \right) - \overline{\Sigma}_{n} \\
\label{decsigman} & + \overline{\Sigma}_{n} - \Sigma .
\end{align}
In order to prove Theorem \ref{vitesselpps}, we just have to give the rates of convergence of the terms on the right-hand side of previous equality. The following lemma gives the almost sure and the rate of convergence in quadratic mean of the first term on the right-hand side of previous equality.
\begin{lem}\label{majopasbelle0}
Suppose assumptions \textbf{(A1)} to \textbf{(A5a')} and \textbf{(A6b)} hold. Then, for all $\gamma > 0$,
\[
\left\| \Sigma_{n} - \frac{1-\delta}{n^{1-\delta}} \sum_{k=1}^{n} \frac{1}{k^{\delta +s}}e^{ - \frac{k^{1-s}}{1-s}} \left( \sum_{j=1}^{k} e^{\frac{j^{1-s}}{2(1-s)}}\left( m_{j} - m \right) \right) \otimes \left( \sum_{j=1}^{k} e^{\frac{j^{1-s}}{2(1-s)}}\left( m_{j} - m \right) \right) \right\|_{F}^{2} = o \left( \frac{(\ln n)^{\gamma}}{n^{1-s}} \right) \quad a.s.
\]
Moreover, suppose assumption \textbf{(A5b)} holds too. Then,
\[
\mathbb{E}\left[ \left\| \Sigma_{n} - \frac{1-\delta}{n^{1-\delta}} \sum_{k=1}^{n} \frac{1}{k^{\delta +s}}e^{ - \frac{k^{1-s}}{1-s}} \left( \sum_{j=1}^{k} e^{\frac{j^{1-s}}{2(1-s)}}\left( m_{j} - m \right) \right) \otimes \left( \sum_{j=1}^{k} e^{\frac{j^{1-s}}{2(1-s)}}\left( m_{j} - m \right) \right) \right\|_{F}^{2} \right]  = O \left( \frac{1}{n^{1-s}} \right) .
\]
\end{lem}
The proof is given in Appendix. The following lemma gives the almost sure and the rate of convergence in quadratic mean of the second term on the right-hand side of equality (\ref{decsigman}).
\begin{lem}\label{lemmajopasbelle}
Suppose assumptions \textbf{(A1)} to \textbf{(A5a')} and \textbf{(A6b)} hold. Then, for all $\gamma > 0$,
\[
\left\| \frac{1-\delta}{n^{1-\delta}} \sum_{k=1}^{n} \frac{1}{k^{\delta +s}}e^{- \frac{k^{1-s}}{1-s}} \left( \sum_{j=1}^{k} e^{\frac{j^{1-s}}{2(1-s)}}\left( m_{j} - m \right) \right) \otimes \left( \sum_{j=1}^{k} e^{\frac{j^{1-s}}{2(1-s)}}\left( m_{j} - m \right) \right) - \overline{\Sigma}_{n} \right\|_{F}^{2} = o \left( \frac{(\ln n)^{\gamma}}{n^{2(1-s)}}\right) \quad a.s.
\]
Moreover, suppose assumption \textbf{(A5b)} holds too. Then
\[
\mathbb{E}\left[ \left\| \frac{1-\delta}{n^{1-\delta}} \sum_{k=1}^{n} \frac{1}{k^{\delta +s}}e^{- \frac{k^{1-s}}{1-s}} \left( \sum_{j=1}^{k} e^{\frac{j^{1-s}}{2(1-s)}}\left( m_{j} - m \right) \right) \otimes \left( \sum_{j=1}^{k} e^{\frac{j^{1-s}}{2(1-s)}}\left( m_{j} - m \right) \right) - \overline{\Sigma}_{n} \right\|_{F}^{2} \right] = O \left( \frac{1}{n^{2(1-s)}}\right) .
\]
\end{lem}
The proof is given in Appendix. Finally, the following Proposition gives the almost sure and the rate of convergence in quadratic mean of the last term on the right-hand side of equality (\ref{decsigman}).
\begin{prop}\label{vitsigmanbarre}
Suppose assumptions \textbf{(A1)} to \textbf{(A5a')} and \textbf{(A6b)} hold. Then, there is a positive constant $\gamma$ such that
\[
\left\| \overline{\Sigma}_{n} - \Sigma \right\|_{F}^{2} = o \left( \frac{(\ln n)^{\delta}}{n^{1-s}} \right) \quad a.s .
\]
Suppose assumption \textbf{(A5b)} holds too. Then, there is a positive constant $C$ such that for all $n \geq 1$,
\[
\mathbb{E}\left[ \left\| \overline{\Sigma}_{n} - \Sigma \right\|_{F}^{2} \right] \leq \frac{ C}{n^{1-s}}.
\]
\end{prop}
\begin{proof}[Proof of Proposition \ref{vitsigmanbarre}]
Applying equality (\ref{normf}), one can check that 
\begin{align}
\notag \left\| \overline{\Sigma}_{n} - \Sigma \right\|_{F} & \leq \frac{1}{\sum_{k=1}^{n}k^{-\delta}}\sum_{k=1}^{n}\frac{1}{k^{\delta}b_{k}} \left\| A_{1,k} \right\|^{2} + \frac{1}{\sum_{k=1}^{n}k^{-\delta}}\sum_{k=1}^{n}\frac{1}{k^{\delta}b_{k}} \left\| A_{2,k} \right\|^{2} \\
\notag & + 2 \frac{1}{\sum_{k=1}^{n}k^{-\delta}}\sum_{k=1}^{n}\frac{1}{k^{\delta}b_{k}}\left\| A_{1,k} \right\| \left\| A_{2,k} \right\|  + 2\frac{1}{\sum_{k=1}^{n}k^{-\delta}}\sum_{k=1}^{n}\frac{1}{k^{\delta}b_{k}} \left\| A_{1,k} \right\| \left\| M_{k+1} \right\| \\
\label{majbourrinsigmanbarre}& + 2 \frac{1}{\sum_{k=1}^{n}k^{-\delta}}\sum_{k=1}^{n}\frac{1}{k^{\delta}b_{k}} \left\| A_{2,k} \right\| \left\| M_{k+1} \right\| + \left\| \frac{1}{\sum_{k=1}^{n}k^{-\delta}}\sum_{k=1}^{n}\frac{1}{k^{\delta}}\left( \frac{1}{b_{k}}M_{k+1}\otimes M_{k+1}  - \Sigma \right)\right\|_{F} ,
\end{align}
where $A_{1,k},A_{2,k},M_{k+1}$ are defined in (\ref{defiai}). The following Lemma gives the rate of convergence in quadratic mean of the first terms on the right-hand side of previous inequality.
\begin{lem}\label{lempleinmaj} Suppose Assumptions \textbf{(A1)} to \textbf{(A6b)} hold. Then, for all $i,j \in \left\lbrace 1,2 \right\rbrace $,
\begin{align*}
& \mathbb{E}\left[ \left( \frac{1}{\sum_{k=1}^{n}k^{-\delta}}\sum_{k=1}^{n}\frac{1}{k^{\delta}b_{k}} \left\| A_{i,k} \right\| \left\| A_{j,k} \right\| \right)^{2}\right] = o \left( \frac{1}{n^{1-s}} \right) , \\
& \mathbb{E}\left[ \left( \frac{1}{\sum_{k=1}^{n}k^{-\delta}}\sum_{k=1}^{n}\frac{1}{k^{\delta}b_{k}} \left\| A_{i,k} \right\| \left\| M_{k+1} \right\| \right)^{2}\right] = o \left( \frac{1}{n^{1-s}} \right) .
\end{align*}
\end{lem}
\noindent The proof of this lemma as well as its "almost sure version" are given in Appendix.

\medskip

Then, we just have to bound the last term on the right-hand side of inequality (\ref{majbourrinsigmanbarre}). First let us decompose $M_{k+1} \otimes M_{k+1}$ as
\begin{align*}
M_{k+1} \otimes M_{k+1} & = \sum_{j=1}^{k} a_{j}^{2} \Xi_{j+1} \otimes \Xi_{j+1} + \sum_{j=1}^{k} a_{j} \Xi_{j+1}\otimes M_{j}  + \sum_{j=1}^{k} a_{j} \Xi_{j+1}\otimes \left( M_{k+1} - M_{j+1} \right) \\
& + \sum_{j=1}^{k}a_{j} M_{j} \otimes \Xi_{j+1} + \sum_{j=1}^{k} a_{j} \left( M_{k+1} - M_{j+1} \right)\otimes \Xi_{j+1} .
\end{align*}
Note that for all $j$, $M_{j}$ is $\mathcal{F}_{j}$-measurable and $\mathbb{E}\left[ \Xi_{j+1} \otimes M_{j} |\mathcal{F}_{j} \right] = 0$. Moreover,
\begin{align*}
\frac{1}{\sum_{k=1}^{n}k^{-\delta}} & \sum_{k=1}^{n}\frac{1}{k^{\delta}}\left( \frac{1}{b_{k}}M_{k+1}\otimes M_{k+1}  - \Sigma \right)  = \frac{1}{\sum_{k=1}^{n}k^{-\delta}}\sum_{k=1}^{n}\frac{1}{k^{\delta}} \frac{1}{b_{k}}\sum_{j=1}^{k}a_{j}\left( \Xi_{j+1}\otimes \Xi_{j+1}    - \Sigma \right) \\
& + \frac{1}{\sum_{k=1}^{n}k^{-\delta}}\sum_{k=1}^{n}\frac{1}{k^{\delta}} \frac{1}{b_{k}}\sum_{j=1}^{k}a_{j}\Xi_{j+1}\otimes M_{j} + \frac{1}{\sum_{k=1}^{n}k^{-\delta}}\sum_{k=1}^{n}\frac{1}{k^{\delta}} \frac{1}{b_{k}}\sum_{j=1}^{k}a_{j} \xi_{j+1}\otimes \left( M_{k+1} - M_{j+1} \right) \\
& + \frac{1}{\sum_{k=1}^{n}k^{-\delta}}\sum_{k=1}^{n}\frac{1}{k^{\delta}} \frac{1}{b_{k}}\sum_{j=1}^{k}a_{j}M_{j}\otimes \Xi_{j+1} + \frac{1}{\sum_{k=1}^{n}k^{-\delta}}\sum_{k=1}^{n}\frac{1}{k^{\delta}} \frac{1}{b_{k}}\sum_{j=1}^{k}a_{j}\left( M_{k+1} - M_{j+1} \right) \otimes \Xi_{j+1} .
\end{align*}
The end of the proof consists in giving a bound of the quadratic mean of each term on the right-hand side of previous equality. Note that the almost sure rates of convergence are not proven since it is quite analogous.

\medskip

\textbf{Bounding $\mathbb{E}\left[ \left\| \frac{1}{\sum_{k=1}^{n}k^{-\delta}}\sum_{k=1}^{n}\frac{1}{k^{\delta}} \frac{1}{b_{k}}\sum_{j=1}^{k}a_{j}\Xi_{j+1}\otimes M_{j} \right\|_{F}^{2}\right]$.} First, note that
\[
\frac{1}{\sum_{k=1}^{n}k^{-\delta}}\sum_{k=1}^{n}\frac{1}{k^{\delta}} \frac{1}{b_{k}}\sum_{j=1}^{k}a_{j}\Xi_{j+1}\otimes M_{j} = \frac{1}{\sum_{k=1}^{n}k^{-\delta}}\sum_{k=1}^{n} \left( \sum_{j=k}^{n} \frac{1}{k^{\delta}}\frac{1}{b_{k}} \right) a_{k}\Xi_{k+1}\otimes M_{k}.
\]
Moreover, with the help of an integral test for convergence, one can check that there is a positive constant $C$ such that for all positive integers $k\leq n$,
\begin{equation}\label{sumpartbk}
 \sum_{j=k}^{n} \frac{1}{k^{\delta}}\frac{1}{b_{k}}  \leq \frac{C}{k^{\delta}}\exp \left( -\frac{k^{1-s}}{(1-s)} \right) .
\end{equation}
Furthermore, since $\left(  \Xi_{j+1} \otimes M_{j}   \right)_{j}$ is a sequence of martingale differences adapted to the filtration $\left( \mathcal{F}_{j} \right)$, let
\begin{align*}
(*) & := \mathbb{E}\left[ \left\| \frac{1}{\sum_{k=1}^{n}k^{-\delta}}\sum_{k=1}^{n}\frac{1}{k^{\delta}} \frac{1}{b_{k}}\sum_{j=1}^{k}a_{j}\Xi_{j+1}\otimes M_{j}\right\|_{F}^{2}\right] \\
& = \mathbb{E}\left[ \left\| \frac{1}{\sum_{k=1}^{n}k^{-\delta}}\sum_{k=1}^{n} \left( \sum_{j=k}^{n} \frac{1}{k^{\delta}}\frac{1}{b_{k}} \right) a_{k}\Xi_{k+1}\otimes M_{k} \right\|_{F}^{2}\right] \\
& = \left( \frac{1}{\sum_{k=1}^{n}k^{-\delta}}\right)^{2}\sum_{k=1}^{n}  \left( \sum_{j=k}^{n} \frac{1}{k^{\delta}}\frac{1}{b_{k}} \right)^{2} a_{k}^{2}\mathbb{E}\left[ \left\| \Xi_{k+1}\otimes M_{k} \right\|_{F}^{2}\right]
\end{align*}
Then, applying equality (\ref{normf}) and Cauchy-Schwarz's inequality,
\begin{align*}
(*) & \leq \left( \frac{1}{\sum_{k=1}^{n}k^{-\delta}}\right)^{2}\sum_{k=1}^{n} \left( \sum_{j=k}^{n} \frac{1}{k^{\delta}}\frac{1}{b_{k}} \right)^{2} a_{k}^{2}\mathbb{E}\left[ \left\| \Xi_{k+1} \right\|^{2} \left\| M_{k} \right\|^{2} \right] \\
& \leq \left( \frac{1}{\sum_{k=1}^{n}k^{-\delta}}\right)^{2}\sum_{k=1}^{n} \left( \sum_{j=k}^{n} \frac{1}{k^{\delta}}\frac{1}{b_{k}} \right)^{2} a_{k}^{2}\sqrt{\mathbb{E}\left[ \left\| \Xi_{k+1} \right\|^{4}\right] \mathbb{E}\left[ \left\| M_{k} \right\|^{4} \right]}.
\end{align*}
Finally, applying Lemmas \ref{lemtech} and \ref{lemmajxi} as well as inequality (\ref{sumpartbk}),
\begin{align*}
(*) = O \left( \left( \frac{1}{\sum_{k=1}^{n}k^{-\delta}}\right)^{2}\sum_{k=1}^{n} \frac{1}{k^{2\delta - s}} \right) = O \left( \frac{1}{n^{1-s}} \right) . 
\end{align*}  
With analogous calculus, one can check
\[
\mathbb{E}\left[ \left\| \frac{1}{\sum_{k=1}^{n}k^{-\delta}}\sum_{k=1}^{n}\frac{1}{k^{\delta}} \frac{1}{b_{k}}\sum_{j=1}^{k}a_{j}M_{j}\otimes \Xi_{j+1} \right\|_{F}^{2}\right]  = O \left( \frac{1}{n^{1-s}}\right) .
\] 

\medskip

\textbf{Bounding $\mathbb{E}\left[ \left\| \frac{1}{\sum_{k=1}^{n}k^{-\delta}}\sum_{k=1}^{n}\frac{1}{k^{\delta}} \frac{1}{b_{k}}\sum_{j=1}^{k}a_{j} \Xi_{j+1} \otimes \left( M_{k+1}-M_{j+1}\right)  \right\|_{F}^{2} \right]$.} First, note that
\begin{align*}
\sum_{j=1}^{k}a_{j} \Xi_{j+1} \otimes \left( M_{k+1} - M_{j} \right) & = \sum_{j=1}^{k}\sum_{j'=j+1}^{k} a_{j}a_{j'}\Xi_{j+1} \otimes \Xi_{j'+1} \\
& = \sum_{j'=2}^{k}\sum_{j=1}^{j'-1}a_{j}a_{j'} \Xi_{j+1}\otimes \Xi_{j'+1} .
\end{align*}
Note that $\left( \sum_{j=1}^{j'-1}a_{j}a_{j'} \Xi_{j+1}\otimes \Xi_{j'+1} \right)_{j'}$ is a sequence of martingale differences adapted to the filtration $\left( \mathcal{F}_{j'} \right)$. Furthermore, 
\begin{align*}
\mathbb{E} &\left[ \left\| \frac{1}{\sum_{k=1}^{n}k^{-\delta}}\sum_{k=1}^{n}\frac{1}{k^{\delta}} \frac{1}{b_{k}}\sum_{j=1}^{k}a_{j} \Xi_{j+1} \otimes \left( M_{k+1}-M_{j}\right)  \right\|_{F}^{2} \right] \\
& = \left( \frac{1}{\sum_{k=1}^{n}k^{-\delta}}\right)^{2}\sum_{k=1}^{n}\frac{1}{k^{2\delta}} \frac{1}{b_{k}^{2}}\mathbb{E}\left[ \left\| \sum_{j'=2}^{k}\sum_{j=1}^{j'-1}a_{j}a_{j'} \Xi_{j+1}\otimes \Xi_{j'+1}  \right\|_{F}^{2} \right] \\
& + \left( \frac{1}{\sum_{k=1}^{n}k^{-\delta}}\right)^{2}\mathbb{E}\left[ \sum_{k=2}^{n}\sum_{j=1}^{k-1}b_{k}^{-1}k^{-\delta}b_{j}^{-1}j^{-\delta} \left\langle \sum_{j''=2}^{j}\sum_{j'=1}^{j''-1}a_{j'}a_{j''}\Xi_{j'+1}\otimes \Xi_{j''+1} , \sum_{i''=2}^{k}\sum_{i'=1}^{i''-1}a_{i'}a_{i''}\Xi_{i'+1}\otimes \Xi_{i''+1} \right\rangle \right] .
\end{align*}
Then end of the proof consists in bounding the two terms on the right-hand side of previous equality. First, since $\left( \sum_{j=1}^{j'-1}a_{j}a_{j'} \Xi_{j+1}\otimes \Xi_{j'+1} \right)_{j'}$ is a sequence of martingale differences adapted to the filtration $\left( \mathcal{F}_{j'} \right)$, let
\begin{align*}
(\star ) & := \left( \frac{1}{\sum_{k=1}^{n}k^{-\delta}}\right)^{2}\sum_{k=1}^{n}\frac{1}{k^{2\delta}} \frac{1}{b_{k}^{2}}\mathbb{E}\left[ \left\| \sum_{j'=2}^{k}\sum_{j=1}^{j'-1}a_{j}a_{j'} \Xi_{j+1}\otimes \Xi_{j'+1}  \right\|_{F}^{2} \right] \\
& = \left( \frac{1}{\sum_{k=1}^{n}k^{-\delta}}\right)^{2}\sum_{k=1}^{n}\frac{1}{k^{2\delta}} \frac{1}{b_{k}^{2}} \sum_{j'=2}^{k}\mathbb{E}\left[ \left\|\sum_{j=1}^{j'-1}a_{j}a_{j'} \Xi_{j+1}\otimes \Xi_{j'+1}  \right\|_{F}^{2} \right] .
\end{align*}
Then, applying equality (\ref{normf}) and Cauchy-Schwarz's inequality,
\begin{align*}
(\star ) & = \left( \frac{1}{\sum_{k=1}^{n}k^{-\delta}}\right)^{2}\sum_{k=1}^{n}\frac{1}{k^{2\delta}} \frac{1}{b_{k}^{2}} \sum_{j'=2}^{k}a_{j'}^{2}\mathbb{E}\left[ \left\|\sum_{j=1}^{j'-1}a_{j} \Xi_{j+1}\right\|^{2} \left\| \Xi_{j'+1}  \right\|^{2} \right] \\
& \leq \left( \frac{1}{\sum_{k=1}^{n}k^{-\delta}}\right)^{2}\sum_{k=1}^{n}\frac{1}{k^{2\delta}} \frac{1}{b_{k}^{2}} \sum_{j'=2}^{k}a_{j'}^{2}\sqrt{\mathbb{E}\left[ \left\| \Xi_{j'+1}  \right\|^{4} \right]}\sqrt{ \mathbb{E}\left[ \left\|\sum_{j=1}^{j'-1}a_{j} \Xi_{j+1}\right\|^{4}\right]  }
\end{align*}
Finally, applying Lemma \ref{lemmajxi}, \ref{lemsumexp} and \ref{lemtech},
\begin{align*}
(\star ) & = O \left( \left( \frac{1}{\sum_{k=1}^{n}k^{-\delta}}\right)^{2}\sum_{k=1}^{n}\frac{1}{k^{2\delta}} \frac{1}{b_{k}^{2}} \sum_{j'=2}^{k}a_{j'}^{4}j'^{s}    \right)   = O \left( \left( \frac{1}{\sum_{k=1}^{n}k^{-\delta}}\right)^{2}\sum_{k=1}^{n}\frac{1}{k^{2\delta}} \frac{1}{b_{k}^{2}}a_{k}^{4}k^{2s} \right) = O \left( \frac{1}{n^{\min \left\lbrace 2-2\delta , 1\right\rbrace }} \right) .
\end{align*}
Then, since $\delta < (1+s) /2$,
\[
\left( \frac{1}{\sum_{k=1}^{n}k^{-\delta}}\right)^{2}\sum_{k=1}^{n}\frac{1}{k^{2\delta}} \frac{1}{b_{k}^{2}}\mathbb{E}\left[ \left\| \sum_{j'=2}^{k}\sum_{j=1}^{j'-1}a_{j}a_{j'} \Xi_{j+1}\otimes \Xi_{j'+1}  \right\|_{F}^{2} \right] = o \left( \frac{1}{n^{1-s}} \right) .
\]
In the same way, by linearity, let
\begin{align*}
 & (\star \star )  := \left( \frac{1}{\sum_{k=1}^{n}k^{-\delta}}\right)^{2}\mathbb{E}\left[ \sum_{k=2}^{n}\sum_{j=1}^{k-1}b_{k}^{-1}k^{-\delta}b_{j}^{-1}j^{-\delta} \left\langle \sum_{j''=2}^{j}\sum_{j'=1}^{j''-1}a_{j'}a_{j''}\Xi_{j'+1}\otimes \Xi_{j''+1} , \sum_{i''=2}^{k}\sum_{i'=1}^{i''-1}a_{i'}a_{i''}\Xi_{i'+1}\otimes \Xi_{i''+1} \right\rangle \right]  \\
& = \left( \frac{1}{\sum_{k=1}^{n}k^{-\delta}}\right)^{2}\sum_{k=2}^{n}\sum_{j=1}^{k-1}b_{k}^{-1}k^{-\delta}b_{j}^{-1}j^{-\delta}\mathbb{E}\left[  \left\langle \sum_{j''=2}^{j}\sum_{j'=1}^{j''-1}a_{j'}a_{j''}\Xi_{j'+1}\otimes \Xi_{j''+1} , \sum_{i''=2}^{j}\sum_{i'=1}^{i''-1}a_{i'}a_{i''}\Xi_{i'+1}\otimes \Xi_{i''+1} \right\rangle_{F} \right] \\
& + \left( \frac{1}{\sum_{k=1}^{n}k^{-\delta}}\right)^{2}\sum_{k=2}^{n}\sum_{j=1}^{k-1}b_{k}^{-1}k^{-\delta}b_{j}^{-1}j^{-\delta} \mathbb{E}\left[ \left\langle \sum_{j''=2}^{j}\sum_{j'=1}^{j''-1}a_{j'}a_{j''}\Xi_{j'+1}\otimes \Xi_{j''+1} , \sum_{i''=j+1}^{k}\sum_{i'=1}^{i''-1}a_{i'}a_{i''}\Xi_{i'+1}\otimes \Xi_{i''+1} \right\rangle_{F} \right] . 
\end{align*}
Since $\left( \Xi_{i''} \right)$ is a sequence of martingale differences adapted to the filtration $\left( \mathcal{F}_{i''} \right)$, 
\begin{align*}
\sum_{k=2}^{n} & \sum_{j=1}^{k-1}b_{k}^{-1}k^{-\delta}b_{j}^{-1}j^{-\delta} \mathbb{E}\left[ \left\langle \sum_{j''=2}^{j}\sum_{j'=1}^{j''-1}a_{j'}a_{j''}\Xi_{j'+1}\otimes \Xi_{j''+1} , \sum_{i''=j+1}^{k}\sum_{i'=1}^{i''-1}a_{i'}a_{i''}\Xi_{i'+1}\otimes \Xi_{i''+1} \right\rangle_{F} \right]  \\
& = \sum_{k=2}^{n} \sum_{j=1}^{k-1}b_{k}^{-1}k^{-\delta}b_{j}^{-1}j^{-\delta}\sum_{j''=2}^{j}\sum_{j'=1}^{j''-1}\sum_{i''=j+1}^{k}\sum_{i'=1}^{i''-1}a_{i'}a_{i''}a_{j'}a_{j''} \mathbb{E}\left[ \left\langle \Xi_{j'+1}\otimes \Xi_{j''+1} , \Xi_{i'+1}\otimes \Xi_{i''+1} \right\rangle_{F} \right] \\
& = \sum_{k=2}^{n} \sum_{j=1}^{k-1}b_{k}^{-1}k^{-\delta}b_{j}^{-1}j^{-\delta}\sum_{j''=2}^{j}\sum_{j'=1}^{j''-1}\sum_{i''=j+1}^{k}\sum_{i'=1}^{i''-1}a_{i'}a_{i''}a_{j'}a_{j''} \mathbb{E}\left[ \left\langle \Xi_{j'+1}\otimes \Xi_{j''+1} , \Xi_{i'+1}\otimes \mathbb{E}\left[ \Xi_{i''+1} |\mathcal{F}_{i''} \right] \right\rangle_{F} \right] \\
& = 0 .
\end{align*}
Furthermore, since $\left( \sum_{j''=2}^{j}\sum_{j'=1}^{j''-1}a_{j'}a_{j''}\Xi_{j'+1}\otimes \Xi_{j''+1} \right)_{j''}$ is a sequence of martingale differences adapted to the filtration $\left( \mathcal{F}_{j''} \right)$ and applying equality (\ref{normf}),
\begin{align*}
(\star \star ) & = \left( \frac{1}{\sum_{k=1}^{n}k^{-\delta}}\right)^{2}\sum_{k=2}^{n}\sum_{j=1}^{k}b_{k}^{-1}k^{-\delta}b_{j}^{-1}j^{-\delta} \mathbb{E}\left[ \left\|  \sum_{j''=1}^{j}\sum_{j'=1}^{j''-1}a_{j'}a_{j''}\Xi_{j'+1}\otimes \Xi_{j''+1} \right\|_{F}^{2} \right] \\
& = \left( \frac{1}{\sum_{k=1}^{n}k^{-\delta}}\right)^{2}\sum_{k=2}^{n}\sum_{j=1}^{k}b_{k}^{-1}k^{-\delta}b_{j}^{-1}j^{-\delta}   \sum_{j''=1}^{j} \mathbb{E}\left[ \left\| \sum_{j'=1}^{j''-1}a_{j'}a_{j''}\Xi_{j'+1}\otimes \Xi_{j''+1} \right\|_{F}^{2} \right] \\
& = \left( \frac{1}{\sum_{k=1}^{n}k^{-\delta}}\right)^{2}\sum_{k=2}^{n}\sum_{j=1}^{k}b_{k}^{-1}k^{-\delta}b_{j}^{-1}j^{-\delta}   \sum_{j''=1}^{j}a_{j''}^{2} \mathbb{E}\left[ \left\| \sum_{j'=1}^{j''-1}a_{j'}\Xi_{j'+1}\right\|^{2} \left\| \Xi_{j''+1} \right\|^{2} \right] .
\end{align*}
Applying Cauchy-Schwarz's inequality as well as Lemmas \ref{lemmajxi} and \ref{lemtech},
\begin{align*}
(\star \star ) & \leq \left( \frac{1}{\sum_{k=1}^{n}k^{-\delta}}\right)^{2}\sum_{k=1}^{n}\sum_{j=1}^{k}b_{k}^{-1}k^{-\delta}b_{j}^{-1}j^{-\delta}   \sum_{j''=1}^{j}a_{j''}^{2} \sqrt{\mathbb{E}\left[ \left\| \sum_{j'=1}^{j''-1}a_{j'}\Xi_{j'+1}\right\|_{F}^{4} \right] \mathbb{E}\left[ \left\| \Xi_{j''+1} \right\|_{F}^{4} \right] } \\
& = O \left( \left( \frac{1}{\sum_{k=1}^{n}k^{-\delta}}\right)^{2}\sum_{k=1}^{n}\sum_{j=1}^{k}b_{k}^{-1}k^{-\delta}b_{j}^{-1}j^{-\delta}\sum_{j''=1}^{j}a_{j''}^{4}j''^{s} \right) . 
\end{align*}
Finally, applying Lemma \ref{lemsumexp},
\begin{align*}
(\star \star ) & = O \left( \left( \frac{1}{\sum_{k=1}^{n}k^{-\delta}}\right)^{2}\sum_{k=1}^{n}\sum_{j=1}^{k}b_{k}^{-1}k^{-\delta}b_{j}^{-1}j^{-\delta} a_{j}^{4}j^{2s} \right) \\
& = O \left( \left( \frac{1}{\sum_{k=1}^{n}k^{-\delta}}\right)^{2}\sum_{k=1}^{n}b_{k}^{-1}k^{-2\delta}k^{2s}a_{k}^{2}\right) \\
& = O \left( \frac{1}{n^{1-s}}    \right) .
\end{align*}
Thus, 
\[
\mathbb{E}\left[ \left\| \frac{1}{\sum_{k=1}^{n}k^{-\delta}}\sum_{k=1}^{n}\frac{1}{k^{\delta}} \frac{1}{b_{k}}\sum_{j=1}^{k}a_{j} \Xi_{j+1} \otimes \left( M_{k+1}-M_{j+1}\right)\right\|_{F}^{2} \right] = O \left( \frac{1}{n^{1-s}}\right) .
\]
Moreover, with analogous calculus, one can check
\[
\mathbb{E}\left[ \left\| \frac{1}{\sum_{k=1}^{n}k^{-\delta}}\sum_{k=1}^{n}\frac{1}{k^{\delta}} \frac{1}{b_{k}}\sum_{j=1}^{k}a_{j} \left( M_{k+1}-M_{j+1}\right) \otimes \Xi_{j+1} \right\|_{F}^{2} \right] = O \left( \frac{1}{n^{1-s}}\right) .
\]

\medskip

\textbf{Bounding $\frac{1}{\sum_{k=1}^{n}k^{-\delta}} \sum_{k=1}^{n} \frac{1}{k^{\delta}b_{k}} \sum_{j=1}^{k}a_{k}^{2} \left( \Xi_{k+1} \otimes \Xi_{k+1} - \Sigma \right)$.} First , note that
\begin{align*}
\frac{1}{\sum_{k=1}^{n}k^{-\delta}} \sum_{k=1}^{n} \frac{1}{k^{\delta}b_{k}} \sum_{j=1}^{k}a_{k}^{2} \left( \Xi_{k+1} \otimes \Xi_{k+1} - \Sigma \right) & = \frac{1}{\sum_{k=1}^{n}k^{-\delta}} \sum_{k=1}^{n} \frac{1}{k^{\delta}b_{k}} \sum_{j=1}^{k}a_{k}^{2} \left( \mathbb{E}\left[ \Xi_{k+1} \otimes \Xi_{k+1} |\mathcal{F}_{k} \right] - \Sigma \right) \\
& + \frac{1}{\sum_{k=1}^{n}k^{-\delta}} \sum_{k=1}^{n} \frac{1}{k^{\delta}b_{k}} \sum_{j=1}^{k}a_{k}^{2} \left( \Xi_{k+1} \otimes \Xi_{k+1} - \mathbb{E}\left[ \Xi_{k+1} \otimes \Xi_{k+1} |\mathcal{F}_{k} \right] \right)
\end{align*}
The end of the proof consists in bounding the quadratic mean of the terms on the right-hand side of previous equality. First, applying Lemma \ref{lemsum}, let 
\begin{align*}
(\star ) & := \mathbb{E}\left[ \left\| \frac{1}{\sum_{k=1}^{n}k^{-\delta}} \sum_{k=1}^{n} \frac{1}{k^{\delta}b_{k}} \sum_{j=1}^{k}a_{j}^{2} \left( \mathbb{E}\left[ \Xi_{k+1} \otimes \Xi_{k+1} |\mathcal{F}_{k} \right] - \Sigma \right) \right\|_{F}^{2} \right] \\
& \leq \left( \frac{1}{\sum_{k=1}^{n}k^{-\delta}}\right)^{2} \left( \sum_{k=1}^{n} \frac{1}{k^{\delta}b_{k}} \sqrt{\mathbb{E}\left[ \left\| \sum_{j=1}^{k}a_{j}^{2} \left( \mathbb{E}\left[ \Xi_{k+1} \otimes \Xi_{k+1}|\mathcal{F}_{k} \right] - \Sigma \right) \right\|_{F}^{2}\right]} \right)^{2} \\
& \leq \left( \frac{1}{\sum_{k=1}^{n}k^{-\delta}}\right)^{2} \left( \sum_{k=1}^{n} \frac{1}{k^{\delta}b_{k}}  \sum_{j=1}^{k} a_{j}^{2} \sqrt{\mathbb{E}\left[ \left\| \mathbb{E}\left[ \Xi_{k+1} \otimes \Xi_{k+1}|\mathcal{F}_{k} \right] - \Sigma   \right\|_{F}^{2}\right]} \right)^{2}
\end{align*} 
Then, applying inequality (\ref{vitlprm}) and Corollary \ref{lemmajxisigma},
\begin{align*}
( \star ) & = O \left(  \left( \frac{1}{\sum_{k=1}^{n}k^{-\delta}}\right)^{2} \left( \sum_{k=1}^{n} \frac{1}{k^{\delta}b_{k}}  \sum_{j=1}^{k} a_{j}^{2} \sqrt{\mathbb{E}\left[ \left\| m_{n}-m  \right\|^{2} \right]} \right)^{2} \right) \\
& = O \left( \left( \frac{1}{\sum_{k=1}^{n}k^{-\delta}}\right)^{2} \left( \sum_{k=1}^{n} \frac{1}{k^{\delta}b_{k}}  \sum_{j=1}^{k} a_{j}^{2} j^{-\alpha /2} \right)^{2} \right) .
\end{align*}
Furthermore, thanks to Lemma \ref{lemsumexp},
\begin{align*}
(\star ) &  = O \left( \left( \frac{1}{\sum_{k=1}^{n}k^{-\delta}}\right)^{2} \left( \sum_{k=1}^{n} \frac{1}{k^{\delta}b_{k}}  a_{k}^{2}k^{s-\alpha /2} \right)^{2} \right)  = O \left( \left( \frac{1}{\sum_{k=1}^{n}k^{-\delta}}\right)^{2} n^{2 -2\delta - \alpha}  \right)  = O \left( \frac{1}{n^{\alpha}} \right) .
\end{align*}
Thus, since $\alpha > 1/2$,
\[
\mathbb{E}\left[ \left\| \frac{1}{\sum_{k=1}^{n}k^{-\delta}} \sum_{k=1}^{n} \frac{1}{k^{\delta}b_{k}} \sum_{j=1}^{k}a_{j}^{2} \left( \mathbb{E}\left[ \Xi_{k+1} \otimes \Xi_{k+1} |\mathcal{F}_{k} \right] - \Sigma \right) \right\|_{F}^{2} \right] = o \left( \frac{1}{n^{1-s}}\right) .
\]

\medskip

\noindent Moreover, applying Lemma \ref{lemsum}, let
\begin{align*}
(\star \star ) & := \mathbb{E}\left[ \left\| \frac{1}{\sum_{k=1}^{n}k^{-\delta}} \sum_{k=1}^{n} \frac{1}{k^{\delta}b_{k}} \sum_{j=1}^{k}a_{j}^{2} \left( \Xi_{k+1} \otimes \Xi_{k+1} - \mathbb{E}\left[ \Xi_{k+1} \otimes \Xi_{k+1} |\mathcal{F}_{k} \right] \right) \right\|_{F}^{2} \right] \\
& \leq \left( \frac{1}{\sum_{k=1}^{n}k^{-\delta}}\right)^{2}\ \left( \sum_{k=1}^{n} \frac{1}{k^{\delta}b_{k}} \sqrt{\mathbb{E}\left[ \left\| \sum_{j=1}^{k}a_{j}^{2} \left( \Xi_{k+1} \otimes \Xi_{k+1} - \mathbb{E}\left[ \Xi_{k+1} \otimes \Xi_{k+1} |\mathcal{F}_{k} \right] \right) \right\|_{F}^{2} \right]} \right)^{2}.
\end{align*}
Furthermore, since $\left( \mathbb{E}\left[ \Xi_{k+1} \otimes \Xi_{k+1} |\mathcal{F}_{k} \right] - \Xi_{k+1}\otimes \Xi_{k+1} \right)$ is a sequence of martingale differences adapted to the filtration $\left( \mathcal{F}_{k} \right)$ and applying Lemma \ref{lemmajxi},
\begin{align*}
( \star \star ) & \leq \left( \frac{1}{\sum_{k=1}^{n}k^{-\delta}}\right)^{2}\ \left( \sum_{k=1}^{n} \frac{1}{k^{\delta}b_{k}} \sqrt{ \sum_{j=1}^{k}a_{j}^{4} \mathbb{E}\left[ \left\| \left( \Xi_{k+1} \otimes \Xi_{k+1} - \mathbb{E}\left[ \Xi_{k+1} \otimes \Xi_{k+1} |\mathcal{F}_{k} \right] \right) \right\|_{F}^{2} \right]} \right)^{2} \\
& = O \left( \left( \frac{1}{\sum_{k=1}^{n}k^{-\delta}}\right)^{2}\ \left( \sum_{k=1}^{n} \frac{1}{k^{\delta}b_{k}} \sqrt{ \sum_{j=1}^{k}a_{j}^{4} } \right)^{2} \right) .
\end{align*}
Then, applying Lemma \ref{lemsumexp},
\begin{align*}
(\star \star ) & =  O \left( \left( \frac{1}{\sum_{k=1}^{n}k^{-\delta}}\right)^{2}\ \left( \sum_{k=1}^{n} \frac{1}{k^{\delta}b_{k}} a_{k}^{2}k^{s/2}  \right)^{2} \right) =  O \left( \left( \frac{1}{\sum_{k=1}^{n}k^{-\delta}}\right)^{2}\ \left( \sum_{k=1}^{n} k^{- \delta -s/2}  \right)^{2} \right)  = O \left( \frac{1}{n^{2-s}}\right) .
\end{align*}
Finally,
\[ 
\mathbb{E}\left[ \left\| \frac{1}{\sum_{k=1}^{n}k^{-\delta}} \sum_{k=1}^{n} \frac{1}{k^{\delta}b_{k}} \sum_{j=1}^{k}a_{j}^{2} \left( \Xi_{k+1} \otimes \Xi_{k+1} - \mathbb{E}\left[ \Xi_{k+1} \otimes \Xi_{k+1} |\mathcal{F}_{k} \right] \right) \right\|_{F}^{2} \right]   = o \left( \frac{1}{n^{1-s}} \right) ,
\]
which concludes the proof.
\end{proof}

\begin{appendix}

\section{Proof of Theorem \ref{tlcgrad}}
Let us recall that the Robbins-Monro algorithm can be written for all $n \geq 1$ as (see (\ref{decbeta}))
\begin{equation*}
 m_{n} - m = \beta_{n-1} \left( m_{1} - m \right)  - \beta_{n-1} \sum_{k=1}^{n-1}\gamma_{k} \beta_{k}^{-1}\delta_{k} + \beta_{n-1} \sum_{k=1}^{n-1}\gamma_{k} \beta_{k}^{-1}\xi_{k+1} .
\end{equation*}
It was proven in \cite{godichon2016} that under assumptions \textbf{(A1)} to \textbf{(A5a)}, for all $\gamma > 0$,
\begin{align*}
\frac{1}{\sqrt{\gamma_{n}}}\left\| \beta_{n-1} \left( m_{1} - m \right)  - \beta_{n-1} \sum_{k=1}^{n-1} \gamma_{k}\beta_{k}^{-1}\delta_{k} \right\| & = O \left( \frac{\left\| m_{n} - m \right\|^{2}}{\sqrt{\gamma_{n}}} \right) \quad a.s , \\
& = o \left( \frac{ ( \ln n )^{\gamma}}{n^{\alpha /2}} \right) \quad a.s.
\end{align*}
Then, we just have to apply Theorem 5.1 in \cite{Jak88} to the last term on the right-hand side of equality (\ref{decbeta}). More precisely, let $\left( e_{i}\right)_{i \in I}$ be an orthonormal basis of $H$ composed of eigenvectors of $\Gamma_{m}$ and let $\psi_{i,j}' := \left\langle \Sigma_{RM}e_{i} , e_{j} \right\rangle$ for all $i,j \in I$, we have to prove that the following equalities are verified.
\begin{equation}
\label{in1rm} \forall \eta > 0, \quad \lim_{n \to \infty} \mathbb{P}\left( \sup_{1\leq k \leq n } \frac{1}{\sqrt{\gamma_{n}}} \left\| \beta_{n}\beta_{k}^{-1}\gamma_{k}\xi_{k+1} \right\| > \eta \right) = 0 ,
\end{equation}
\begin{equation}
\label{in2rm} \lim_{n \to \infty}\frac{1}{\gamma_{n}}\sum_{k=1}^{n} \left\langle \beta_{n}  \beta_{k}^{-1}\gamma_{k}\xi_{k+1} , e_{i} \right\rangle \left\langle \beta_{n}  \beta_{k}^{-1}\gamma_{k}\xi_{k+1} , e_{j} \right\rangle = \psi_{i,j}' \quad a.s, \quad \forall i,j \in I ,
\end{equation}
\begin{equation}\label{in3rm}
\forall \epsilon > 0 , \quad \lim_{N \to \infty } \limsup_{n \to \infty} \mathbb{P} \left( \frac{1}{\gamma_{n}}\sum_{k=1}^{n} \sum_{j=N}^{\infty} \left\langle \beta_{n} \beta_{k}^{-1}\gamma_{k}\xi_{k+1} , e_{j} \right\rangle ^{2} > \epsilon \right) = 0.
\end{equation}

\medskip

\textbf{Proof of (\ref{in1rm}).} Let $\eta > 0$, applying Markov's inequality, 
\begin{align*}
\mathbb{P}\left( \sup_{1\leq k \leq n} \frac{1}{\sqrt{\gamma_{n}}}\left\| \beta_{n}\beta_{k}^{-1}\gamma_{k} \xi_{k+1} \right\| > \eta \right) & \leq \sum_{k=1}^{n} \mathbb{P}\left( \frac{1}{\sqrt{\gamma_{n}}}\left\| \beta_{n} \beta_{k}^{-1} \gamma_{k} \xi_{k+1} \right\| > \eta \right) \\ & \leq \frac{1}{\eta^{4} \gamma_{n}^{2}} \sum_{k=1}^{n} \mathbb{E}\left[  \left\| \beta_{n} \beta_{k}^{-1} \gamma_{k}\xi_{k+1} \right\|^{4} \right] .
\end{align*}
First, since each eigenvalue $\lambda$ of $\Gamma_{m}$ verifies $0 < \lambda_{\min} \leq \lambda \leq C$, there is a rank $n_{\alpha}$ such that for all positive integer $k,n$ verifying $n_{\alpha} \leq k \leq n$,
\begin{align}\label{majbeta}
\left\| \beta_{n}\beta_{k}^{-1} \right\|_{op} \leq \prod_{j=k+1}^{n}\left\| I_{H} - \gamma_{j}\Gamma_{m} \right\|_{op} \leq \prod_{j=k+1}^{n}\left( 1- \gamma_{j} \lambda_{\min} \right) \leq \exp \left( - \lambda_{\min}\sum_{j=k+1}^{n} \gamma_{j} \right) .  
\end{align}
For the sake of simplicity, we consider from now that $n_{\alpha} = 1$ (one can see the proof of Lemma 3.1 in \cite{CCG2015} for an analogous and more detailed proof). Then, applying Lemmas \ref{lemmajxi} and \ref{sumexppart}, there is a positive constant $C$ such that for all $n \geq 1$,
\begin{align*}
\sum_{k=1}^{n-1} \mathbb{E}\left[  \left\| \beta_{n} \beta_{k}^{-1} \gamma_{k}\xi_{k+1} \right\|^{4} \right] &  \leq \sum_{k=1}^{n}\left\| \beta_{n}\beta_{k}^{-1} \right\|_{op}^{4}\gamma_{k}^{4}\mathbb{E}\left[ \left\| \xi_{k+1} \right\|^{4} \right]  \\
& \leq C \sum_{k=1}^{n} \exp \left( - 4 \lambda_{\min}\sum_{j=k+1}^{n}\gamma_{j} \right) \gamma_{k}^{4} \\
& = O \left( \gamma_{n}^{3}\right) ,
\end{align*} 
which concludes the proof of (\ref{in1rm}).

\medskip

\textbf{Proof of (\ref{in2rm}).} Since
\[
\sum_{k=1}^{n} \left\langle \beta_{n}\beta_{k}^{-1}\gamma_{k}\xi_{k+1},e_{i} \right\rangle  \left\langle \beta_{n}\beta_{k}^{-1}\gamma_{k}\xi_{k+1},e_{j} \right\rangle  = \sum_{k=1}^{n} \left\langle \left( \beta_{n}\beta_{k}^{-1}\gamma_{k}\xi_{k+1}\right) \otimes \left( \beta_{n}\beta_{k}^{-1}\gamma_{k}\xi_{k+1}\right) \left( e_{i} \right) , e_{j} \right\rangle ,
\]
we just have to prove that
\begin{equation}
\label{convbeta} \lim_{n \to \infty} \left\| \frac{1}{\gamma_{n}} \sum_{k=1}^{n} \left( \beta_{n}\beta_{k}^{-1}\gamma_{k}\xi_{k+1}\right) \otimes \left( \beta_{n}\beta_{k}^{-1}\gamma_{k}\xi_{k+1}\right) - \Sigma_{RM} \right\|_{F} = 0 \quad a.s.
\end{equation}
 First, note that by linearity
\begin{align*}
\sum_{k=1}^{n} \left( \beta_{n}\beta_{k}^{-1}\gamma_{k}\xi_{k+1} \right) \otimes \left( \beta_{n}\beta_{k}^{-1}\gamma_{k}\xi_{k+1} \right) & = \sum_{k=1}^{n}\left( \beta_{n}\beta_{k}^{-1}\gamma_{k}\right) \left(  \xi_{k+1} \otimes \xi_{k+1}\right)\left(\beta_{n}\beta_{k}^{-1}\gamma_{k} \right) \\
& = \sum_{k=1}^{n}\left( \beta_{n}\beta_{k}^{-1}\gamma_{k}\right)  \mathbb{E}\left[ \xi_{k+1} \otimes \xi_{k+1} |\mathcal{F}_{k} \right] \left(\beta_{n}\beta_{k}^{-1}\gamma_{k} \right) \\
& + \sum_{k=1}^{n}\left( \beta_{n}\beta_{k}^{-1}\gamma_{k}\right)\epsilon_{k+1}\left(\beta_{n}\beta_{k}^{-1}\gamma_{k} \right) ,
\end{align*}
with $\epsilon_{k+1} = \xi_{k+1} \otimes \xi_{k+1} - \mathbb{E}\left[ \xi_{k+1} \otimes \xi_{k+1} |\mathcal{F}_{k} \right]$. Note that $\left( \epsilon_{k} \right)$ is a sequence of martingale differences adapted to the filtration $\left( \mathcal{F}_{k} \right)$. We now prove that the two last terms on the right-hand side of previous equality converge almost surely to $0$. First, as in \cite{godichon2016} and \cite{CCG2015}, one can check that 
\[
\lim_{n\to \infty}\frac{1}{\gamma_n}\left\| \sum_{k=1}^{n} \left( \beta_{n}\beta_{k}^{-1}\gamma_{k}\right)\epsilon_{k+1}\left(\beta_{n}\beta_{k}^{-1}\gamma_{k} \right) \right\|_{F} = 0 \quad a.s.
\]
Let us now rewrite $ \mathbb{E}\left[ \xi_{k+1} \otimes \xi_{k+1} |\mathcal{F}_{k} \right]$  as 
\begin{align*}
\mathbb{E}\left[ \xi_{k+1} \otimes \xi_{k+1} |\mathcal{F}_{k} \right] = \Sigma ' + \left( \mathbb{E}\left[ \nabla_{h}g \left( X_{k+1} , m_{k} \right) \otimes \nabla_{h}g \left( X_{k+1} , m_{k} \right) |\mathcal{F}_{k} \right] - \Sigma' \right) - \Phi \left( m_{k}\right) \otimes \Phi \left( m_{k} \right) .
\end{align*}
Then, let
\begin{align*}
(*)  & := \frac{1}{\gamma_{n}}\left\| \sum_{k=1}^{n} \left( \beta_{n} \beta_{k}^{-1}\gamma_{k} \right) \left( \Phi (m_{k} ) \otimes \Phi \left( m_{k} \right) \right) \left( \beta_{n} \beta_{k}^{-1}\gamma_{k} \right) \right\|_{F} \\
 &  \leq \frac{1}{\gamma_{n}} \sum_{k=1}^{n} \left\| \left( \beta_{n} \beta_{k}^{-1}\gamma_{k} \right) \left( \Phi (m_{k} ) \otimes \Phi \left( m_{k} \right) \right) \left( \beta_{n} \beta_{k}^{-1}\gamma_{k} \right) \right\|_{F} \\
 & \leq \frac{1}{\gamma_{n}} \sum_{k=1}^{n} \left\| \beta_{n} \beta_{k}^{-1}\gamma_{k} \Phi \left( m_{k} \right) \right\|^{2} .
\end{align*}
Moreover, since there is a positive constant $C$ such that for all $n \geq 1$, $\left\| \Phi (m_{n} ) \right\| \leq C\| m_{n} - m \|$,
\begin{align*}
(*) & \leq \frac{1}{\gamma_{n}}\sum_{k=1}^{n} \left\| \beta_{n} \beta_{k}^{-1} \right\|_{op}^{2}\gamma_{k}^{2} \left\| \Phi \left( m_{k} \right) \right\|^{2}  \leq  \frac{1}{\gamma_{n}}\sum_{k=1}^{n} \left\| \beta_{n} \beta_{k}^{-1} \right\|_{op}^{2}\gamma_{k}^{2} C^{2} \left\| m_{k}-m \right\|^{2}.
\end{align*}
Thus, applying inequalities (\ref{vitesseasrm}) and (\ref{majbeta}) as well as Lemma \ref{sumexppart}, for all $\beta < \alpha$,
\begin{align*}
\frac{1}{\gamma_{n}}\sum_{k=1}^{n} \left\| \beta_{n}\beta_{k}^{-1} \right\|_{op}^{2}\gamma_{k}^{2}C^{2} \left\| m_{k} - m \right\|^{2} & = o \left( \frac{1}{\gamma_{n}}\sum_{k=1}^{n} \exp \left( - 2\lambda_{\min} \sum_{j=k+1}^{n}\gamma_{j} \right)\gamma_{k}^{2}\frac{1}{k^{\beta}}\right)   = o \left( \frac{1}{n^{\beta}} \right) . 
\end{align*}
\medskip
In the same way,
\begin{align*}
\frac{1}{\gamma_{n}} &  \left\| \sum_{k=1}^{n} \left( \beta_{n} \beta_{k}^{-1}\gamma_{k} \right)\left( \mathbb{E}\left[ \nabla_{h}g \left( X_{k+1} , m_{k} \right) \otimes \nabla_{h}g \left( X_{k+1} , m_{k} \right) |\mathcal{F}_{k} \right] - \Sigma' \right) \left( \beta_{n} \beta_{k}^{-1}\gamma_{k} \right) \right\|_{F} \\
& \leq \frac{1}{\gamma_{n}}  \sum_{k=1}^{n} \left\|  \beta_{n} \beta_{k}^{-1}\gamma_{k} \right\|_{op}^{2}\left\| \mathbb{E}\left[ \nabla_{h}g \left( X_{k+1} , m_{k} \right) \otimes \nabla_{h}g \left( X_{k+1} , m_{k} \right) |\mathcal{F}_{k} \right] - \Sigma' \right\|_{F} \\
& \leq \frac{1}{\gamma_{n}}  \sum_{k=1}^{n} \gamma_{k}^{2} \exp \left( - 2 \lambda_{\min} \sum_{j=k+1}^{n}\gamma_{j} \right) \left\| \mathbb{E}\left[ \nabla_{h}g \left( X_{k+1} , m_{k} \right) \otimes \nabla_{h}g \left( X_{k+1} , m_{k} \right) |\mathcal{F}_{k} \right] - \Sigma' \right\|_{F} .
\end{align*}
Then, with the help of assumption \textbf{(A6a)}, Lemma \ref{sumexppart} and Toeplitz's lemma, one can check that
\[
\lim_{n \to \infty} \frac{1}{\gamma_{n}} \left\| \sum_{k=1}^{n} \left( \beta_{n} \beta_{k}^{-1}\gamma_{k} \right)\left( \mathbb{E}\left[ \nabla_{h}g \left( X_{k+1} , m_{k} \right) \otimes \nabla_{h}g \left( X_{k+1} , m_{k} \right) |\mathcal{F}_{k} \right] - \Sigma' \right) \left( \beta_{n} \beta_{k}^{-1}\gamma_{k} \right) \right\|_{F} = 0 \quad a.s.
\]
\medskip
In order to verify equality (\ref{convbeta}), we have to prove 
\[
\lim_{n \to \infty}\frac{1}{\gamma_{n}}\left\| \sum_{k=1}^{n} \gamma_{k}^{2} \beta_{n}\beta_{k}^{-1} \Sigma ' \beta_{n} \beta_{k}^{-1} - \Sigma_{RM} \right\| = 0.
\]
Let $\left( e_{i} \right)_{i \in I}$ be an orthonormal basis of $H$ composed of eigenvectors of $\Gamma_{m}$, and let $\left( \lambda_{i} \right)_{i \in I}$ be the set of the associated eigenvalues. Then, let us rewrite $\nabla_{h}g \left( X , m \right)$ as
\[
\nabla_{h}g \left( X , m \right) = \sum_{i \in I}\left\langle \nabla_{h}g \left( X , m \right) , e_{i} \right\rangle e_{i}   ,
\] 
and it comes, by linearity and by dominated convergence, 
\begin{align*}
\frac{1}{\gamma_{n}}& \sum_{k=1}^{n} \gamma_{k}^{2}  \beta_{n}\beta_{k}^{-1}  \Sigma ' \beta_{n}\beta_{k}^{-1}  \\
& = \frac{1}{\gamma_{n}}\sum_{k=1}^{n}\gamma_{k}^{2}\mathbb{E}\left[ \beta_{n}\beta_{k}^{-1}  \nabla_{h}g  \left( X , m \right) \otimes \beta_{n}\beta_{k}^{-1}  \nabla_{h}g  \left( X , m \right) \right] \\
& = \mathbb{E}\left[\frac{1}{\gamma_{n}}\sum_{k=1}^{n}\gamma_{k}^{2} \left(  \sum_{i \in I}\left\langle \nabla_{h}g \left( X , m \right) , e_{i} \right\rangle \prod_{j=k+1}^{n} \left( 1-\gamma_{k}\lambda_{i} \right)e_{i} \right) \otimes \left( \sum_{i \in I}\left\langle \nabla_{h}g \left( X , m \right) , e_{i} \right\rangle \prod_{j=k+1}^{n} \left( 1-\gamma_{k}\lambda_{i} \right)e_{i} \right) \right] .
\end{align*}
In the same way,
\begin{align*}
\Sigma_{RM} & = \int_{0}^{\infty} e^{-sH} \Sigma ' e^{-sH} ds \\
& = \int_{0}^{\infty}\mathbb{E}\left[ \left(\sum_{i \in I}\left\langle \nabla_{h}g \left( X , m \right) , e_{i} \right\rangle e^{-\lambda_{i}s} e_{i} \right) \otimes \left(\sum_{i \in I}\left\langle \nabla_{h}g \left( X , m \right) , e_{i} \right\rangle e^{-\lambda_{i}s} e_{i} \right) \right] ds \\
& = \mathbb{E}\left[\int_{0}^{\infty} \left(\sum_{i \in I}\left\langle \nabla_{h}g \left( X , m \right) , e_{i} \right\rangle e^{-\lambda_{i}s} e_{i} \right) \otimes \left(\sum_{i \in I}\left\langle \nabla_{h}g \left( X , m \right) , e_{i} \right\rangle e^{-\lambda_{i}s} e_{i} \right)  ds \right] .
\end{align*}
In order to conclude the proof, let us now introduce the following lemma, which allows to give a bound of $ \left\| \frac{1}{\gamma_{n}} \sum_{k=1}^{n}   \beta_{n}\beta_{k}^{-1} \gamma_{k} \Sigma ' \beta_{n}\beta_{k}^{-1} \gamma_{k} - \Sigma_{RM} \right\|_{F}$.
\begin{lem}\label{lempelunif}
There is a positive sequence $\left( a_{n} \right)$ such that for all $n \geq 1$ and for all $ i,i' \in I$,
\[
-a_{n} \leq \frac{1}{\gamma_{n}}\sum_{k=1}^{n}\gamma_{k}^{2}\prod_{j=k+1}^{n} \left( 1-\gamma_{j}\lambda_{i} \right) \left( 1-\gamma_{j}\lambda_{i'} \right)- \int_{0}^{\infty} e^{-\left( \lambda_{i} + \lambda_{i'} \right)s }ds \leq a_{n},
\]
and $\lim_{n \to \infty} a_{n} = 0$.
\end{lem}
\begin{proof}
The proof is given in Appendix. 
\end{proof}
Thanks to previous lemma, let
\begin{align*}
& (*)  = \left\| \frac{1}{\gamma_{n}}\sum_{k=1}^{n}   \beta_{n}\beta_{k}^{-1} \gamma_{k} \Sigma ' \beta_{n}\beta_{k}^{-1} \gamma_{k} - \Sigma_{RM}    \right\|_{F} \\
& \leq \mathbb{E}\left[\sqrt{\sum_{i ,i' \in I} \left( \frac{1}{\gamma_{n}}\sum_{k=1}^{n}\gamma_{k}^{2} \prod_{j=k+1}^{n} \left( 1-\gamma_{k}\lambda_{i} \right) \left( 1-\gamma_{k}\lambda_{i'} \right)- \frac{1}{\lambda_{i}+ \lambda_{i'}} \right)^{2}\left\langle \nabla_{h}g \left( X , m \right) , e_{i} \right\rangle^{2}\left\langle \nabla_{h}g \left( X , m \right) , e_{i'} \right\rangle^{2} } \right]  \\
& \leq a_{n} \mathbb{E}\left[ \sqrt{\sum_{i,i' \in I} \left\langle \nabla_{h}g \left( X , m \right) , e_{i} \right\rangle^{2}\left\langle \nabla_{h}g \left( X , m \right) , e_{i'} \right\rangle^{2}} \right] \\
& = a_{n} \mathbb{E}\left[ \sum_{i \in I} \left\langle \nabla_{h}g \left( X , m \right) , e_{i} \right\rangle^{2}\right] .
\end{align*}
Under assumption \textbf{(A5a)},
\begin{align*}
\left\| \frac{1}{\gamma_{n}}\sum_{k=1}^{n}   \beta_{n}\beta_{k}^{-1} \gamma_{k} \Sigma ' \beta_{n}\beta_{k}^{-1} \gamma_{k} - \Sigma_{RM}    \right\|_{F} & \leq  a_{n} \mathbb{E}\left[ \sum_{i \in I} \left\langle \nabla_{h}g \left( X , m \right) , e_{i} \right\rangle^{2}\right] \\
& =  a_{n}\mathbb{E}\left[ \left\| \nabla_{h} g \left( X,m \right) \right\|^{2} \right] \\
& \leq L_{1}a_{n}.
\end{align*}
Since $a_{n}$ converges to $0$, this concludes the proof of inequality (\ref{in2rm}).

\medskip

\textbf{Proof of inequality (\ref{in3rm})} Let $\epsilon > 0$, applying Markov's inequality,
\begin{align*}
\mathbb{P} & \left( \frac{1}{\gamma_{n}} \sum_{k=1}^{n} \sum_{j=N}^{\infty} \left\langle \beta_{n}\beta_{k}^{-1}\gamma_{k}\xi_{k+1}, e_{j} \right\rangle > \epsilon  \right)  \leq \frac{1}{\gamma_{n}\epsilon^{2}}\sum_{k=1}^{n} \sum_{j=N}^{\infty} \mathbb{E}\left[ \left\langle \beta_{n}\beta_{k}^{-1}\gamma_{k}\xi_{k+1}, e_{j} \right\rangle^{2} \right] \\
& = \frac{1}{\gamma_{n}\epsilon^{2}} \sum_{k=1}^{n} \sum_{j=N}^{\infty} \mathbb{E}\left[ \mathbb{E}\left[ \left\langle \beta_{n}\beta_{k}^{-1}\xi_{k+1}\otimes \beta_{n}\beta_{k}^{-1}\xi_{k+1} |\mathcal{F}_{k} \right] (e_{j}), e_{j} \right\rangle^{2} \right]  \\
& = \frac{1}{\epsilon^{2}}  \sum_{j=N}^{\infty}\frac{1}{\gamma_{n}}\sum_{k=1}^{n} \mathbb{E}\left[ \mathbb{E}\left[ \left\langle \beta_{n}\beta_{k}^{-1}\xi_{k+1}\otimes \beta_{n}\beta_{k}^{-1}\xi_{k+1} |\mathcal{F}_{k} \right] (e_{j}), e_{j} \right\rangle^{2} \right]
\end{align*}
Since $\frac{1}{\gamma_n}\sum_{k=1}^{n} \left( \beta_{n} \beta_{k}^{-1}\gamma_{k} \xi_{k+1} \right) \otimes \left( \beta_{n} \beta_{k}^{-1}\gamma_{k} \xi_{k+1} \right) $ converges almost surely to $\Sigma_{RM}$ with respect to the Frobenius norm and by dominated convergence,
\[
\limsup_{n} \mathbb{P} \left( \frac{1}{\gamma_{n}} \sum_{k=1}^{n} \sum_{j=N}^{\infty} \left\langle \beta_{n}\beta_{k}^{-1}\gamma_{k}\xi_{k+1},e_{j} \right\rangle > \epsilon \right) \leq \frac{1}{\epsilon^2}\sum_{j=N}^{\infty} \left\langle \Sigma_{RM}(e_{j}) ,e_{j} \right\rangle .
\]
Moreover, since
\begin{align*}
\sum_{j= 1}^{\infty} \left\langle \Sigma_{RM} (e_{j}) ,e_{j} \right\rangle = \left\| \int_{0}^{\infty} e^{-sH} \Sigma ' e^{-sH}ds  \right\|_{F}  \leq \frac{1}{2\lambda_{\min}} \left\| \Sigma ' \right\|_{F} \leq \frac{L_{1}}{2\lambda_{\min}},
\end{align*}
and since $\left\langle \Sigma_{RM}\left( e_{j} \right) , e_{j} \right\rangle \geq 0$ for all $j \in I$, 
\[
\lim_{N \to \infty} \frac{1}{\epsilon^{2}}\sum_{j=N}^{\infty} \left\langle \Sigma_{RM}\left( e_{j} \right) , e_{j} \right\rangle = 0,
\]
which concludes the proof.

\section{Proof of Lemma \ref{lemtech}}
\begin{proof}

\textbf{Bounding $ \mathbb{E}\left[ \left\| \sum_{k=1}^{n}\frac{a_{k}}{\gamma_{k}} \left( T_{k} - T_{k+1} \right) \right\|^{2p} \right]$.} Applying an Abel's transform,
\[
A_{1,n} = \frac{a_{1}}{\gamma_{1}}T_{1} - \frac{a_{n}}{\gamma_{n}}T_{n+1} + \sum_{k=2}^{n}\left( \frac{a_{k}}{\gamma_{k}} - \frac{a_{k-1}}{\gamma_{k-1}}\right) T_{k} .
\] 
First, $\mathbb{E}\left[ \left\| \frac{a_{1}}{\gamma_{1}}T_{1} \right\|^{2p} \right] = O \left( 1 \right)$. Moreover, applying inequality (\ref{vitlprm}) ,
\begin{align*}
\mathbb{E}\left[ \left\| \frac{a_{n}}{\gamma_{n}}T_{n+1} \right\|^{2p}\right] & \leq \exp \left( \frac{pn^{1-s}}{1-s} \right)c_{\gamma}^{-1}n^{2p\alpha} \frac{C_{p}\lambda_{\min}^{-2p}}{n^{p\alpha}} \\
& \leq \exp \left( \frac{pn^{1-s}}{1-s} \right)C_{p}c_{\gamma}^{-1}\lambda_{\min}^{-2p}n^{p\alpha}.
\end{align*}
Furthermore, one can check that there is a positive constant $C$ such that for all $n \geq 1$, 
\[
\left| \frac{a_{n}}{\gamma_{n}} - \frac{a_{n-1}}{\gamma_{n-1}} \right| \leq C n^{-s+\alpha}\exp \left( \frac{n^{1-s}}{2(1-s)}\right), 
\]
and applying Lemma \ref{lemsum} and inequality (\ref{vitlprm}),
\begin{align*}
\mathbb{E}\left[ \left\| \sum_{k=2}^{n}\left( \frac{a_{k}}{\gamma_{k}} - \frac{a_{k-1}}{\gamma_{k-1}}\right) T_{k} \right\|^{2p} \right] & \leq \left( \sum_{k=2}^{n}  \left| \frac{a_{k}}{\gamma_{k}} - \frac{a_{k-1}}{\gamma_{k-1}} \right| \left( \mathbb{E}\left[ \left\| T_{k} \right\|^{2p} \right] \right)^{\frac{1}{2p}}\right)^{2p} \\
& \leq C^{2p}C_{p}\lambda_{\min}^{-2p} \left( \sum_{k=2}^{n} k^{-s+\alpha}\exp \left( \frac{k^{1-s}}{2(1-s)} \right)k^{-\alpha / 2} \right)^{2p} .
\end{align*}
Finally, applying Lemma \ref{lemsumexp}, 
\[
\mathbb{E}\left[ \left\| \sum_{k=2}^{n}\left( \frac{a_{k}}{\gamma_{k}} - \frac{a_{k-1}}{\gamma_{k-1}}\right) T_{k} \right\|^{2p} \right] = O \left( \exp \left( \frac{pn^{1-s}}{1-s} \right) n^{p\alpha } \right) .
\]

\medskip

\textbf{Bounding $\mathbb{E}\left[ \left\| \sum_{k=1}^{n} a_{k} \Delta_{k} \right\|^{2p}\right]$.} Since there is a positive constant $C_{m}$ (see \cite{godichon2016}) such that for all $n \geq 1$, $\left\| \Delta_{n} \right\| \leq C_{m}\left\| T_{n} \right\|^{2}$, applying Lemma \ref{lemsum} and inequality (\ref{vitlprm}),
\begin{align*}
\mathbb{E}\left[ \left\| \sum_{k=1}^{n} a_{k} \Delta_{k} \right\|^{2p} \right] & \leq C_{m}^{2p}\lambda_{\min}^{-4p} \left( \sum_{k=1}^{n} a_{k} \left( \mathbb{E}\left[ \left\| m_{n} - m \right\|^{4p} \right] \right)^{\frac{1}{2p}} \right)^{2p} \\
& \leq C_{m}^{2p}\lambda_{\min}^{-4p}C_{2p} \left( \sum_{k=1}^{n} a_{k}k^{-\alpha}  \right)^{2p}  
\end{align*}
Applying Lemma \ref{lemsumexp},
\[
\mathbb{E}\left[ \left\| \sum_{k=1}^{n} a_{k} \Delta_{k} \right\|^{2p} \right] = O \left( \exp \left( \frac{pn^{1-s}}{1-s} \right) n^{2p (s- \alpha ) }\right)
\]

\medskip

\textbf{Bounding $\mathbb{E}\left[ \left\| \sum_{k=1}^{n}a_{k}\Xi_{k+1} \right\|^{2p}\right]$.} First, since $\left( \Xi_{n} \right)$ is a sequence of martingale differences, and thanks to Lemma \ref{lemsumexp}, 
\begin{align*}
\mathbb{E}\left[ \left\| \sum_{k=1}^{n} a_{k}\Xi_{k+1} \right\|^{2} \right] & = \sum_{k=1}^{n} a_{k}^{2}\mathbb{E}\left[ \left\| \Xi_{k+1} \right\|^{2} \right] \\
& = O \left( \exp \left( \frac{n^{1-s}}{1-s} \right) n^{s} \right) .
\end{align*}
With the help of an induction on $p$ (see the proof of Theorem 4.2 in \cite{godichon2015} for instance), one can check that for all integer $p \geq 1$,
\[
\mathbb{E}\left[ \left\| \sum_{k=1}^{n} a_{k} \Xi_{k+1} \right\|^{2p}\right] = O \left( \exp \left( p\frac{n^{1-s}}{1-s} \right) n^{ps} \right) ,
\]
which concludes the proof.
\end{proof}

\section{Proof of Lemma \ref{lempelunif}}
Let $\left( \lambda_{i} \right)_{i \in I}$ be the eigenvalues of the Hessian $\Gamma_{m}$. First, let 
\begin{align*}
c_{n,k} :=  \prod_{j=k+1}^{n} \left( 1-\gamma_{j}\lambda_{i} \right)\prod_{j=k+1}^{n} \left( 1-\gamma_{j}\lambda_{i'} \right) = \exp \left(  \sum_{j=k+1}^{n} \left( \ln \left( 1- \gamma_{j}\lambda_{i} \right) +  \ln \left( 1- \gamma_{j}\lambda_{i'} \right) \right) \right)
\end{align*}
Let us recall that there is a positive constant $C$ such taht for all $i \in I$, $\lambda_{i} \leq C$. Then, let $n_{\alpha}$ be an integer such that for all $k \geq n_{\alpha}$, $C\gamma_{k} < 1$, and it comes, for all $k \geq n_{\alpha}$, $\lambda_{i}\gamma_{k} \leq C \gamma_{k} < 1$. Then, with the help of the Taylor's expansion of the functional $x \longmapsto \ln (1-x)$, one can check that for all $i \in I$ and for all $k \geq n_{\alpha}$,
\[
- \lambda_{i} \gamma_{k} \geq   \ln \left( 1- \lambda_{i}\gamma_{k} \right) \geq - \lambda_{i}\gamma_{k} - \frac{\lambda_{i}^{2}\gamma_{k}^{2}}{1-C\gamma_{n_{\alpha}}} = - \lambda_{i}\gamma_{k} - c\gamma_{k}^{2},
\]
with $c:= \frac{1}{1-C\gamma_{n_{\alpha}}}$. Then, for all $n,k \geq n_{\alpha}$,
\[
\exp \left( - \sum_{j=k+1}^{n}\left(  \left( \lambda_{i}+ \lambda_{i'} \right) \gamma_{j} + 2c\gamma_{j}^{2} \right)\right) \leq c_{n,k} \leq \exp \left( - \sum_{j=k+1}^{n} \left( \lambda_{i}+ \lambda_{i'} \right) \gamma_{j} \right) .
\]
With the help of an integral test for convergence,
\begin{align*}
c_{n,k} & \geq \exp \left( - \left( \lambda_{i} + \lambda_{i'} \right) \gamma_{k+1} - c_{\gamma}\int_{k+1}^{n}  \left( \lambda_{i}+ \lambda_{i'} \right)t^{-\alpha}dt  - 2c\gamma_{k+1}^{2} - 2cc_{\gamma}^{2} \int_{k+1}^{n}t^{-2\alpha}dt \right) \\
c_{n,k} & \leq \exp \left( - \left( \lambda_{i} + \lambda_{i'} \right) \gamma_{k+1} - c_{\gamma}\int_{k+1}^{n} \left( \lambda_{i}+ \lambda_{i'} \right) t^{-\alpha}dt \right) .
\end{align*}
Then,
\begin{align*}
c_{n,k}  \geq & \exp \left( - \left( \lambda_{i} + \lambda_{i'} \right) \left( \frac{c_{\gamma}}{1-\alpha}\left( (k+1)^{1-\alpha} - n^{1-\alpha} \right) - \gamma_{k+1}\right) \right) \\
&  \times \exp \left( 2c\left( \gamma_{k+1}^{2} -\frac{ c_{\gamma}^{2}}{1-2\alpha} \left( (n^{1-2\alpha} -(k+1)^{1-2\alpha} \right)  \right) \right)    \\
c_{n,k}  \leq &  \exp \left( -\left( \lambda_{i} + \lambda_{i'} \right) \left( \frac{c_{\gamma}}{1-\alpha}\left( (k+1)^{1-\alpha} - n^{1-\alpha} \right) - \gamma_{k+1}\right) \right)
\end{align*}
We now give an upper bound of $\sum_{k=1}^{n} \gamma_{k}^{2}c_{n,k}$. Since $0 < \lambda_{\min} \leq \lambda_{i} \leq C$ for all $i \in I$, there is a rank $n_{\alpha}$, only depending on $\lambda_{\min} , C, c_{\gamma}$ and $\alpha$, such that the functional $\varphi : \mathbb{R} \longrightarrow \mathbb{R}$ defined for all $t \in \mathbb{R}$ by
\[
\varphi (t) := c_{\gamma}^{2}t^{-2\alpha} \exp \left( -\left( \lambda_{i} + \lambda_{i'} \right) \left( \frac{c_{\gamma}}{1-\alpha}\left( (t+1)^{1-\alpha} - n^{1-\alpha} \right) - c_{\gamma}(t+1)^{-\alpha}\right) \right) ,
\]
is increasing on $[n_{\alpha} , + \infty ]$. For the sake of simplicity, let us consider that $n_{\alpha} = 0$. Then, with the help of an integral test for convergence, 
\begin{align}
\notag \sum_{k=1}^{n} &  \gamma_{k}^{2}c_{n,k}  \leq \int_{0}^{n}c_{\gamma}^{2}t^{-2\alpha} \exp \left( -\left( \lambda_{i} + \lambda_{i'} \right) \left( \frac{c_{\gamma}}{1-\alpha}\left( (t+1)^{1-\alpha} - n^{1-\alpha} \right) - c_{\gamma}(t+1)^{-\alpha}\right) \right) dt \\
\notag & = c_{\gamma}\frac{1}{\lambda_{i} + \lambda_{i'}} \left[ \exp \left( -\left( \lambda_{i} + \lambda_{i'} \right) \frac{c_{\gamma}}{1-\alpha}\left( (t+1)^{1-\alpha} - n^{1-\alpha} \right) \right)t^{-\alpha}\exp \left( -\left( \lambda_{i} + \lambda_{i'} \right)  c_{\gamma}(t+1)^{- \alpha} \right)   \right]_{0}^{n} \\
\label{inegintpourri} &  + c_{\gamma}\frac{1}{\lambda_{i} + \lambda_{i'}} \int_{0}^{n} e^{ -\left( \lambda_{i} + \lambda_{i'} \right) \frac{c_{\gamma}}{1-\alpha}\left( (t+1)^{1-\alpha} - n^{1-\alpha} \right) } \left(  t^{-1-\alpha} - \left( \lambda_{i} + \lambda_{i'} \right)  c_{\gamma}t^{-2\alpha}\right)   e^{ -\left( \lambda_{i} + \lambda_{i'} \right)  c_{\gamma}(t+1)^{- \alpha} } dt
\end{align}
Then, since for all $i \in I$, $0< \lambda_{\min} \leq \lambda_{i} \leq C$, one can check that there is a positive sequence $\left( \epsilon_{n} \right)_{n \geq 1}$ only depending on $\alpha, c_{\gamma}, \lambda_{\min}, C$ such that
\begin{align*}
& \sum_{k=1}^{n}\gamma_{k}^{2}c_{n,k} \leq \frac{\gamma_{n}}{\lambda_{i} + \lambda_{i'}} + \epsilon_{n}\gamma_{n}, \quad \quad \text{and} \quad \quad \lim_{n \to \infty} \epsilon_{n} = 0.
\end{align*}
With analogous calculus, on can check that there is a positive sequence $\left( \epsilon_{n} ' \right)_{n \geq 1}$ only depending on $\alpha , c_{\gamma}, \lambda_{\min} , C$ such that
\begin{align*}
& \sum_{k=1}^{n} \gamma_{k}^{2}c_{n,k} \geq \frac{\gamma_{n}}{\lambda_{i} + \lambda_{i'}} - \epsilon_{n}'\gamma_{n},  \quad \quad \text{and} \quad \quad \lim_{n \to \infty} \epsilon_{n}' = 0,
\end{align*}
which concludes the proof.

\section{Proof of Lemma \ref{majopasbelle0}}
We only give the bound of the quadratic mean error since the almost sure rate of convergence is quite straightforward. First, since
\begin{align*}
\left( m_{j} - \overline{m}_{j} \right) \otimes \left( m_{j} - \overline{m}_{j} \right) & - \left( m_{j} - m \right) \otimes \left( m_{j} - m \right) \\
& =  \left( m_{j} - m + m - \overline{m}_{j} \right) \otimes \left( m_{j}- m + m - \overline{m}_{j} \right) - \left( m_{j} - m \right) \otimes \left( m_{j} - m \right) \\
& = \left( m - \overline{m}_{j} \right) \otimes \left( m_{j} - m \right) + \left( m_{j} - m \right) \otimes \left( m - \overline{m}_{j} \right) + \left( m - \overline{m}_{j} \right) \otimes \left( m - \overline{m}_{j} \right) ,
\end{align*} 
and by linearity, let
\begin{align*}
(\star ) & := \Sigma_{n} - \frac{1-\delta}{n^{1-\delta}} \sum_{k=1}^{n} \frac{1}{k^{\delta +s}}\exp \left( - \frac{k^{1-s}}{1-s}\right) \left( \sum_{j=1}^{k} e^{\frac{j^{1-s}}{2(1-s)}}\left( m_{j} - m \right) \right) \otimes \left( \sum_{j=1}^{k} e^{\frac{j^{1-s}}{2(1-s)}}\left( m_{j} - m \right) \right) \\
& = - \frac{1-\delta}{n^{1-\delta}} \sum_{k=1}^{n} \frac{1}{k^{\delta +s}}\exp \left( - \frac{k^{1-s}}{1-s}\right) \left( \sum_{j=1}^{k} e^{\frac{j^{1-s}}{2(1-s)}}\left( m_{j} - m \right) \right) \otimes \left( \sum_{j=1}^{k} e^{\frac{j^{1-s}}{2(1-s)}}\left( \overline{m}_{j} - m \right) \right) \\
& - \frac{1-\delta}{n^{1-\delta}} \sum_{k=1}^{n} \frac{1}{k^{\delta +s}}\exp \left( - \frac{k^{1-s}}{1-s}\right) \left( \sum_{j=1}^{k} e^{\frac{j^{1-s}}{2(1-s)}}\left( \overline{m}_{j} - m \right) \right) \otimes \left( \sum_{j=1}^{k} e^{\frac{j^{1-s}}{2(1-s)}}\left( m_{j} - m \right) \right) \\
& + \frac{1-\delta}{n^{1-\delta}} \sum_{k=1}^{n} \frac{1}{k^{\delta +s}}\exp \left( - \frac{k^{1-s}}{1-s}\right) \left( \sum_{j=1}^{k} e^{\frac{j^{1-s}}{2(1-s)}}\left( \overline{m}_{j} - m \right) \right) \otimes \left( \sum_{j=1}^{k} e^{\frac{j^{1-s}}{2(1-s)}}\left( \overline{m}_{j} - m \right) \right) .
\end{align*}
Then, we have to bound the three terms on the right-hand side of previous equality. 

\medskip

\textbf{Bounding $\mathbb{E}\left[ \left\| \frac{1-\delta}{n^{1-\delta}} \sum_{k=1}^{n} \frac{1}{k^{\delta +s}}\exp \left( - \frac{k^{1-s}}{1-s}\right) \left( \sum_{j=1}^{k} e^{\frac{j^{1-s}}{2(1-s)}}\left( m_{j} - m \right) \right) \otimes \left( \sum_{j=1}^{k} e^{\frac{j^{1-s}}{2(1-s)}}\left( \overline{m}_{j} - m \right) \right) \right\|_{F}^{2}\right] $.} First, applying Lemma \ref{lemsum} and equality (\ref{normf}), let
\begin{align*}
(*) & := \mathbb{E}\left[ \left\| \frac{1-\delta}{n^{1-\delta}} \sum_{k=1}^{n} \frac{1}{k^{\delta +s}}\exp \left( - \frac{k^{1-s}}{1-s}\right) \left( \sum_{j=1}^{k} e^{\frac{j^{1-s}}{2(1-s)}}\left( m_{j} - m \right) \right) \otimes \left( \sum_{j=1}^{k} e^{\frac{j^{1-s}}{2(1-s)}}\left( \overline{m}_{j} - m \right) \right) \right\|_{F}^{2}\right] \\
& \leq \left( \frac{1-\delta}{n^{1-\delta}}\right)^{2} \left(  \sum_{k=1}^{n} \frac{1}{k^{\delta +s}}\exp \left( - \frac{k^{1-s}}{1-s}\right) \sqrt{\mathbb{E}\left[ \left\|  \left( \sum_{j=1}^{k} e^{\frac{j^{1-s}}{2(1-s)}}\left( m_{j} - m \right) \right) \otimes \left( \sum_{j=1}^{k} e^{\frac{j^{1-s}}{2(1-s)}}\left( \overline{m}_{j} - m \right) \right) \right\|_{F}^{2} \right]}\right)^{2} \\
& \leq \left( \frac{1-\delta}{n^{1-\delta}}\right)^{2} \left(  \sum_{k=1}^{n} \frac{1}{k^{\delta +s}}\exp \left( - \frac{k^{1-s}}{1-s}\right) \sqrt{\mathbb{E}\left[ \left\|   \sum_{j=1}^{k} e^{\frac{j^{1-s}}{2(1-s)}}\left( m_{j} - m \right)  \right\|^{2} \left\|  \sum_{j=1}^{k} e^{\frac{j^{1-s}}{2(1-s)}}\left( \overline{m}_{j} - m \right)  \right\|^{2} \right]}\right)^{2} .
\end{align*}
Applying Cauchy-Schwarz's inequality,
\begin{align*}
(*) \leq  & \left( \frac{1-\delta}{n^{1-\delta}}\right)^{2}  \left(  \sum_{k=1}^{n} \frac{1}{k^{\delta +s}}e^{ - \frac{k^{1-s}}{1-s}} \left(\mathbb{E}\left[ \left\|   \sum_{j=1}^{k} e^{\frac{j^{1-s}}{2(1-s)}}\left( m_{j} - m \right)  \right\|^{4} \right] \right)^{\frac{1}{4}} \left( \mathbb{E}\left[ \left\|   \sum_{j=1}^{k} e^{\frac{j^{1-s}}{2(1-s)}}\left( \overline{m}_{j} - m \right)  \right\|^{4}\right]\right)^{\frac{1}{4}} \right)^{2}.
\end{align*}
First, note that thanks to Lemma \ref{lemtech}
\[
\mathbb{E}\left[ \left\|   \sum_{j=1}^{k} e^{\frac{j^{1-s}}{2(1-s)}}\left( m_{j} - m \right)  \right\|^{4} \right] = O \left( \exp \left( \frac{2k^{1-s}}{1-s}\right) k^{2s} \right) .
\]
Furthermore, applying Lemmas \ref{lemsum} and Lemma \ref{lemsumexp} as well as inequality (\ref{vitlpmoy}),
\begin{align}
\notag \mathbb{E}\left[ \left\|   \sum_{j=1}^{k} e^{\frac{j^{1-s}}{2(1-s)}}\left( \overline{m}_{j} - m \right)  \right\|^{4}\right] & \leq \left(  \sum_{j=1}^{k} e^{\frac{j^{1-s}}{2(1-s)}}\left( \mathbb{E}\left[ \left\| \overline{m}_{j} - m \right\|^{4} \right]\right)^{\frac{1}{4}}\right)^{4} \\
\notag & \leq C_{2}' \left(  \sum_{j=1}^{k} e^{\frac{j^{1-s}}{2(1-s)}}\frac{1}{j^{1/2}}\right)^{4} \\
\label{nouvmaj} & = O \left(  \exp \left( 2\frac{k^{1-s}}{(1-s)} \right)k^{4s}k^{-2} \right) .
\end{align}
Then, applying Lemma \ref{lemsumexp},
\begin{align*}
(*) & = O \left( \left( \frac{1-\delta}{n^{1-\delta}}\right)^{2} \left(  \sum_{k=1}^{n} \frac{1}{k^{\delta + 1/2 - s/2}} \right)^{2}\right)  = O \left( \frac{1}{n^{1-s}}\right) .
\end{align*}
\medskip

With analogous calculus, one can check that

\[
\mathbb{E}\left[ \left\| \frac{1-\delta}{n^{1-\delta}} \sum_{k=1}^{n} \frac{1}{k^{\delta +s}}e^{ - \frac{k^{1-s}}{1-s}} \left( \sum_{j=1}^{k} e^{\frac{j^{1-s}}{2(1-s)}}\left( \overline{m}_{j} - m \right) \right) \otimes \left( \sum_{j=1}^{k} e^{\frac{j^{1-s}}{2(1-s)}}\left( m_{j} - m \right) \right) \right\|_{F}^{2}\right] = O \left( \frac{1}{n^{1-s}} \right) .
\]

\medskip

\textbf{Bounding $\mathbb{E}\left[ \left\| \frac{1-\delta}{n^{1-\delta}} \sum_{k=1}^{n} \frac{1}{k^{\delta +s}}\exp \left( - \frac{k^{1-s}}{1-s}\right) \left( \sum_{j=1}^{k} e^{\frac{j^{1-s}}{2(1-s)}}\left( \overline{m}_{j} - m \right) \right) \otimes \left( \sum_{j=1}^{k} e^{\frac{j^{1-s}}{2(1-s)}}\left( \overline{m}_{j} - m \right) \right) \right\|_{F}^{2}\right]$.} First, applying Lemma \ref{lemsum} and equality (\ref{normf}), let
\begin{align*}
(**) & = \mathbb{E}\left[ \left\| \frac{1-\delta}{n^{1-\delta}} \sum_{k=1}^{n} \frac{1}{k^{\delta +s}}\exp \left( - \frac{k^{1-s}}{1-s}\right) \left( \sum_{j=1}^{k} e^{\frac{j^{1-s}}{2(1-s)}}\left( \overline{m}_{j} - m \right) \right) \otimes \left( \sum_{j=1}^{k} e^{\frac{j^{1-s}}{2(1-s)}}\left( \overline{m}_{j} - m \right) \right) \right\|_{F}^{2}\right] \\
& \leq \left( \frac{1-\delta}{n^{1-\delta}} \sum_{k=1}^{n} \frac{1}{k^{\delta +s}} e^{ - \frac{k^{1-s}}{1-s}} \sqrt{\mathbb{E}\left[ \left\| \left( \sum_{j=1}^{k} e^{\frac{j^{1-s}}{2(1-s)}}\left( \overline{m}_{j} - m \right) \right) \otimes \left( \sum_{j=1}^{k} e^{\frac{j^{1-s}}{2(1-s)}}\left( \overline{m}_{j} - m \right) \right) \right\|_{F}^{2}\right]} \right)^{2} \\
& =  \left( \frac{1-\delta}{n^{1-\delta}} \sum_{k=1}^{n} \frac{1}{k^{\delta +s}} e^{ - \frac{k^{1-s}}{1-s}} \sqrt{\mathbb{E}\left[ \left\|  \sum_{j=1}^{k} e^{\frac{j^{1-s}}{2(1-s)}}\left( \overline{m}_{j} - m  \right) \right\|_{F}^{4}\right]} \right)^{2}.
\end{align*}
Then, applying inequality (\ref{nouvmaj}) and Corollary \ref{corbn},
\begin{align*}
(**) = O \left(  \left( \frac{1-\delta}{n^{1-\delta}} \sum_{k=1}^{n} \frac{1}{k^{1+\delta -s}}  \right)^{2} \right) = O \left( \frac{1}{n^{2(1-s)}} \right) ,
\end{align*}
which concludes the proof.

\section{Proof of Lemma \ref{lemmajopasbelle}}
We just give the proof for the rate of convergence in quadratic mean, the proof of the almost sure rate of convergence is quite straightforward. Let
\begin{align*}
(\star ) & := \frac{1-\delta}{n^{1-\delta}} \sum_{k=1}^{n} \frac{1}{k^{\delta +s}}\exp \left( - \frac{k^{1-s}}{1-s}\right) \left( \sum_{j=1}^{k} e^{\frac{j^{1-s}}{2(1-s)}}\left( m_{j} - m \right) \right) \otimes \left( \sum_{j=1}^{k} e^{\frac{j^{1-s}}{2(1-s)}}\left( m_{j} - m \right) \right) - \overline{\Sigma}_{n} \\
& = \left( \frac{1-\delta}{n^{1-\delta}} - \frac{1}{\sum_{k=1}^{n} k^{-\delta}} \right)\sum_{k=1}^{n} \frac{1}{k^{\delta +s}}\exp \left( - \frac{k^{1-s}}{1-s}\right) \left( \sum_{j=1}^{k} e^{\frac{j^{1-s}}{2(1-s)}}\left( m_{j} - m \right) \right) \otimes \left( \sum_{j=1}^{k} e^{\frac{j^{1-s}}{2(1-s)}}\left( m_{j} - m \right) \right) \\
& + \frac{1}{\sum_{k=1}^{n}k^{-\delta}}\sum_{k=1}^{n} \frac{1}{k^{\delta}}\left( k^{-s}\exp \left( - \frac{k^{1-s}}{1-s}\right) - b_{k}^{-1}\right)\left( \sum_{j=1}^{k} e^{\frac{j^{1-s}}{2(1-s)}}\left( m_{j} - m \right) \right) \otimes \left( \sum_{j=1}^{k} e^{\frac{j^{1-s}}{2(1-s)}}\left( m_{j} - m \right) \right)
\end{align*}
We now bound the quadratic mean of each term on the right-hand side of previous equality. First, note that with the help of an integral test for convergence,
\begin{align*}
\frac{1}{1-\delta} \left( (n+1)^{1-\delta}  \right) = \int_{0}^{n+1}t^{-\delta}dt \geq  \sum_{k=1}^{n}k^{-\delta}   \geq \int_{1}^{n}t^{-\delta}dt = \frac{1}{1-\delta} \left( n^{1-\delta} - 1 \right) .  
\end{align*}
Then,
\begin{align*}
\left| \frac{1-\delta}{n^{1-\delta}} - \frac{1}{\sum_{k=1}^{n}k^{-\delta}} \right| & = \left| \frac{ \frac{n^{1-\delta}}{1-\delta} - \sum_{k=1}^{n}k^{-\delta}}{\frac{n^{1-\delta}}{1-\delta}  \sum_{k=1}^{n}k^{-\delta}}\right| \\
& \leq   \left( 1- \delta \right) \frac{(n+1)^{1-\delta} - n^{1-\delta} +1 }{n^{1-\delta}\left( n^{1-\delta}-1 \right)} \\
& = O \left( \frac{1}{n^{2 -2\delta}} \right) .
\end{align*}
Then, applying Lemma \ref{lemsum}, there is a positive constant $C$ such that for all $n \geq 1$,
\begin{align*}
u_{n} & := \mathbb{E}\left[ \left\| \left( \frac{1-\delta}{n^{1-\delta}} - \frac{1}{\sum_{k=1}^{n} k^{-\delta}} \right)\sum_{k=1}^{n} \frac{1}{k^{\delta +s}}e^{ - \frac{k^{1-s}}{1-s}} \left( \sum_{j=1}^{k} e^{\frac{j^{1-s}}{2(1-s)}}\left( m_{j} - m \right) \right) \otimes \left( \sum_{j=1}^{k} e^{\frac{j^{1-s}}{2(1-s)}}\left( m_{j} - m \right) \right) \right\|_{F}^{2} \right] \\
& \leq \frac{C}{n^{4-4\delta}}\left( \sum_{k=1}^{n} \frac{1}{k^{\delta +s}}\exp \left( - \frac{k^{1-s}}{1-s}\right) \sqrt{\mathbb{E}\left\|  \left( \sum_{j=1}^{k} e^{\frac{j^{1-s}}{2(1-s)}}\left( m_{j} - m \right) \right) \otimes \left( \sum_{j=1}^{k} e^{\frac{j^{1-s}}{2(1-s)}}\left( m_{j} - m \right) \right) \right\|_{F}^{2}} \right)^{2} .
\end{align*}
Furthermore, applying equality (\ref{normf})
\begin{align*}
u_{n} & \leq \frac{C}{n^{4-4\delta}}\left( \sum_{k=1}^{n} \frac{1}{k^{\delta +s}}\exp \left( - \frac{k^{1-s}}{1-s}\right) \sqrt{\mathbb{E}\left\|  \sum_{j=1}^{k} e^{\frac{j^{1-s}}{2(1-s)}}\left( m_{j} - m \right)  \right\|_{F}^{4}} \right)^{2}
\end{align*}
Finally, applying Lemma \ref{lemtech} ans since $\delta < (1+s)/2$, 
\begin{align*}
u_{n} & = O \left( \frac{C}{n^{4-4\delta}}\left( \sum_{k=1}^{n} \frac{1}{k^{\delta}}\right)^{2} \right)  =  O \left( \frac{1}{n^{2-2\delta}} \right) = o \left( \frac{1}{n^{1-s}}\right) . 
\end{align*}

\medskip

In the same way, with the help of an integral test for convergence,
\begin{align*}
b_{n}  := \sum_{k=1}^{n} \exp \left( \frac{k^{1-s}}{1-s} \right) &  \leq \int_{0}^{n} \exp \left( \frac{t^{1-s}}{(1-s)} \right) dt \\
&  \leq n^{s}\exp \left( \frac{ n^{1-s}}{(1-s)} \right) + s \int_{0}^{n} \exp \left( \frac{t^{1-s}}{(1-s)}\right) t^{s-1}dt \\
& = n^{s}\exp \left( \frac{ n^{1-s}}{1-s} \right) + sn^{2s-1}\exp \left( \frac{ n^{1-s}}{1-s} \right) + o \left( n^{2s-1}\exp \left( \frac{ n^{1-s}}{1-s} \right) \right) .
\end{align*}
Thus, one can check that there is a positive constant $c$ such that for all $n \geq 1$,
\[
b_{n} \geq n^{s}\exp \left( \frac{ n^{1-s}}{1-s} \right) + cn^{2s-1}\exp \left( \frac{ n^{1-s}}{1-s} \right)
\]
Then,
\begin{align*}
\left| \frac{1}{b_{n}} - n^{-s}\exp \left( - \frac{ n^{1-s}}{1-s} \right) \right| & =  \frac{n^{-s}\exp \left( - \frac{  n^{1-s}}{1-s} \right)}{b_{n}}\left| b_{n} - n^{s}\exp \left(  \frac{ n^{1-s}}{1-s} \right)  \right| \\
& = O \left( n^{-1}\exp \left( - \frac{ n^{1-s}}{1-s} \right) \right)
\end{align*}
Thus, applying Lemma \ref{lemsum}, there is a positive constant $C$ such that for all $n \geq 1$,
\begin{align*}
v_{n} & := \mathbb{E}\left[ \left\| \frac{1}{\sum_{k=1}^{n}k^{-\delta}}\sum_{k=1}^{n} \frac{1}{k^{\delta}}\left( k^{-s}e^{ - \frac{k^{1-s}}{1-s}} - b_{k}^{-1}\right)\left( \sum_{j=1}^{k} e^{\frac{j^{1-s}}{2(1-s)}}\left( m_{j} - m \right) \right) \otimes \left( \sum_{j=1}^{k} e^{\frac{j^{1-s}}{2(1-s)}}\left( m_{j} - m \right) \right) \right\|_{F}^{2} \right] \\
& \leq \left( \frac{1}{\sum_{k=1}^{n}k^{-\delta}} \right)^{2}\left( \sum_{k=1}^{n} \frac{1}{k^{\delta}}k^{-1}e^{ - \frac{ k^{1-s}}{1-s} } \sqrt{ \mathbb{E}\left[ \left\| \left( \sum_{j=1}^{k} e^{\frac{j^{1-s}}{2(1-s)}}\left( m_{j} - m \right) \right) \otimes \left( \sum_{j=1}^{k} e^{\frac{j^{1-s}}{2(1-s)}}\left( m_{j} - m \right) \right) \right\|_{F}^{2} \right]} \right)^{2} .
\end{align*}
Finally, applying equality (\ref{normf}) and Lemma \ref{lemtech},
\begin{align*}
v_{n} & \leq \left( \frac{1}{\sum_{k=1}^{n}k^{-\delta}} \right)^{2}\left( \sum_{k=1}^{n} \frac{1}{k^{\delta}}k^{-1}\exp \left( - \frac{ k^{1-s}}{1-s} \right) \sqrt{ \mathbb{E}\left[ \left\|  \sum_{j=1}^{k} e^{\frac{j^{1-s}}{2(1-s)}}\left( m_{j} - m \right)  \right\|_{F}^{4} \right]} \right)^{2}\\
 & = O \left( \left( \frac{1}{\sum_{k=1}^{n}k^{-\delta}} \right)^{2}\left( \sum_{k=1}^{n} \frac{1}{k^{\delta +1-s}} \right)^{2} \right) \\
& = O \left( \frac{1}{n^{2(1-s)}} \right) ,
\end{align*}
which concludes the proof.

\section{Proof of Lemma \ref{lempleinmaj}}
\begin{proof}[Proof of Lemma \ref{lempleinmaj}]
This proof is a direct application of Lemma \ref{lemtech}. In order to convince the reader, we just give one proof, and the other ones are analogous. Applying Lemma \ref{lemsum} and \ref{lemtech} as well as Corollary \ref{corbn},
\begin{align*}
\mathbb{E}\left[ \left( \frac{1}{\sum_{k=1}^{n}k^{-\delta}}\sum_{k=1}^{n}\frac{1}{k^{\delta}b_{k}}\left\| A_{1,k} \right\|^{2} \right)^{2} \right] & \leq \left(\frac{1}{ \sum_{k=1}^{n}k^{-\delta}}\right)^{2}\left( \sum_{k=1}^{n}\frac{1}{k^{\delta}b_{k}} \sqrt{ \mathbb{E}\left[  \left\| A_{1,k} \right\|^{4} \right]} \right)^{2} \\
& = O \left(  \left(\frac{1}{ \sum_{k=1}^{n}k^{-\delta}}\right)^{2}\left( \sum_{k=1}^{n}\frac{1}{k^{\delta}} k^{\alpha -s}   \right)^{2}\right) \\
& = O \left( \frac{1}{n^{2(s- \alpha)}}\right) ,
\end{align*}
which concludes the proof.
\end{proof}
We now give the "almost sure version" of Lemma \ref{lempleinmaj}.
\begin{lem} Suppose Assumptions \textbf{(A1)} to \textbf{(A5a')} hold. Then, for all $i,j \in \left\lbrace 1,2 \right\rbrace $, and for all $\gamma > 0$,
\begin{align*}
& \mathbb{E}\left[ \left( \frac{1}{\sum_{k=1}^{n}k^{-\delta}}\sum_{k=1}^{n}\frac{1}{k^{\delta}b_{k}} \left\| A_{i,k} \right\| \left\| A_{j,k} \right\| \right)^{2}\right] = o \left( \frac{( \ln n)^{\gamma}}{n^{1-s}} \right) , \\
& \mathbb{E}\left[ \left( \frac{1}{\sum_{k=1}^{n}k^{-\delta}}\sum_{k=1}^{n}\frac{1}{k^{\delta}b_{k}} \left\| A_{i,k} \right\| \left\| M_{k+1} \right\| \right)^{2}\right] = o \left( \frac{(\ln n)^{\gamma}}{n^{1-s}} \right) .
\end{align*}
\end{lem}
The proof is not given since it is quite closed to the one of Lemma \ref{lempleinmaj}.

\section{Dealing with Assumption (A6) for the geometric median}
In what follows, we consider that assumption \textbf{(H2)} in \cite{godichon2016} is fulfilled, i.e: 
\begin{itemize}
\item[\textbf{(H2)}] The random variable $X$ is not concentrated around single points: for all positive constant $A$, there is a positive constant $C_{A}$ such that for all $h \in \mathcal{B}\left( 0 , A \right)$,
\[
\mathbb{E}\left[ \frac{1}{\left\| X-h \right\|^{2}} \right] \leq C_{A}.
\]
\end{itemize}
Then, for all $h \in H$, let us define the function $\varphi_{h} : [ 0,1 ] \longrightarrow \mathcal{S}(H)$, defined for all $t \in [0,1 ]$ by
\begin{align*}
\varphi_{h} (t) & = \mathbb{E}\left[ \nabla_{h}g \left( X , m + t \left( h-m \right) \right) \otimes \nabla_{h}g \left( X , m + t \left( h-m \right) \right) \right] \\
& = \mathbb{E}\left[ \frac{X-m+t \left( h-m \right)}{\left\| X-m+t \left( h-m^{v} \right) \right\|} \otimes \frac{X-m^{v}+t \left( h-m \right)}{\left\| X-m+t \left( h-m \right) \right\|} \right] .
\end{align*}
In what follows, we will denote $A(t):=X-m+t \left( h-m \right)$. Note that 
\begin{align*}
& \varphi_{h}(0) = \mathbb{E}\left[ \nabla_{h}g_{v} \left( X , m \right) \otimes \nabla_{h}g_{v} \left( X , m \right) \right] & \varphi_{h}(1) = \mathbb{E}\left[ \nabla_{h}g_{v} \left( X , h \right) \otimes \nabla_{h}g_{v} \left( X , h \right) \right]
\end{align*}
and that the functional $\varphi_{h}$ is differentiable, and its derivative is defined for all $t \in [0,1]$ by
\begin{align*}
\varphi_{h}'(t) & = -2\mathbb{E}\left[ \frac{1}{\left\| A(t) \right\|^{4} }  \left\langle h-m , A(t) \right\rangle  A(t)  \otimes  A(t)  \right] + \mathbb{E}\left[ \frac{1}{\left\| A(t) \right\|^{2}}\left( h-m \right) \otimes A(t) \right] \\
& +\mathbb{E}\left[ \frac{1}{\left\| A(t) \right\|^{2}} A(t) \otimes \left( h-m \right) \right] .
\end{align*}
Then, applying Cauchy-Schwarz's inequality, 
\begin{align*}
\left\| \varphi_{h}'(t) \right\|_{F} & \leq 4 \mathbb{E}\left[ \frac{1}{\left\| A(t) \right\|}\right] \left\| m - h \right\| .
\end{align*}
Thus, let $\epsilon > 0$, thanks to Assumption \textbf{(H2)}, there is a positive constant $C_{\left\| m \right\| + \epsilon }$ such that for all $t \in [0,1]$ and for all $h \in \mathcal{B}\left( m , \epsilon \right)$, 
\[
\left\| \varphi_{h}'(t) \right\|_{F} \leq C_{\left\| m \right\| + \epsilon } \left\| m - h \right\| .
\]
Finally, 
\begin{align*}
\left\| \mathbb{E}\left[ \nabla_{h}g_{v} \left( X , m \right) \otimes \nabla_{h}g_{v} \left( X , m \right) \right] - \mathbb{E}\left[ \nabla_{h}g_{v} \left( X , h \right) \otimes \nabla_{h}g_{v} \left( X , h \right) \right] \right\|_{F} & = \left\| \varphi_{h}(1) - \varphi_{h}(0) \right\|_{F} \\
& = \left\| \int_{0}^{1}\varphi_{h}'(t) dt \right\|_{F} \\
& \leq C_{\left\| m \right\| + \epsilon } \left\| m - h \right\| .
\end{align*}

\section{Technical lemmas}
In order to simplify the proof, we recall or give some technical lemmas. The following one ensures that the sequence $\left( \xi_{n} \right)$ admits uniformly bounded $2p$-moments.
\begin{lem}[\cite{godichon2016}] \label{lemmajxi}
Suppose assumptions \textbf{(A1)} to \textbf{(A5a')} hold, there is a positive constant $K$ such that for all $n \geq 1 $,
\[
\mathbb{E}\left[ \left\| \xi_{n+1} \right\|^{4} \right] \leq K
\]
Moreover, suppose assumption \textbf{(A5b)} holds too. Then, for all positive integer $p$, there is a positive constant $K_{p}$ such that for all $n \geq 1 $,
\[
\mathbb{E}\left[ \left\| \xi_{n+1} \right\|^{2p} \right] \leq K_{p}
\]
As a particular case, since for all eigenvalue $\lambda$ of $\Gamma_{m}$, $0 < \lambda_{\min} \leq \lambda \leq C$, for all $n \geq 1$,
\[
\mathbb{E}\left[ \left\| \Xi_{n+1} \right\|^{2p} \right] \leq K_{p}\lambda_{\min}^{-2p}.
\]
\end{lem}

\begin{cor}\label{lemmajxisigma} Suppose assumptions \textbf{(A1)} to \textbf{(A6b)} hold. Then, there is a positive constant $C$ such that for all $n \geq 1$,
\[
\left\| \mathbb{E}\left[ \Xi_{n+1} \otimes \Xi_{n+1}|\mathcal{F}_{n} \right] - \Sigma  \right\|_{F}^{2} \leq C \left\| m_{n} - m \right\|^{2}
\]
\end{cor}
The proof is not given since it is a direct application of assumption \textbf{(A6b)} and Lemma~\ref{lemmajxi}. The following lemma gives upper bounds of the sums of exponential terms which appears in several proofs. 
\begin{lem}\label{lemsumexp}
For all constants $a,b,c$ such that $a \in ( 0,1)$, there is a positive constant $C_{a,b,c}$ such that 
\[
\sum_{k=1}^{n}k^{-a}k^{b}\exp\left( ck^{1-a} \right) \leq C_{a,b,c}n^{b}\exp \left(  cn^{1-a} \right) . 
\]
\end{lem}
The proof is not given it is a direct application of an integral test for convergence. As a corollary, one can obtain the following bound (lower and upper) of $b_{n}$.
\begin{cor}\label{corbn}
There are positive constants $c,C$ such that for all $n \geq 1$,
\[
c n^{s}\exp \left( \frac{n^{1-s}}{1-s} \right) \leq b_{n} \leq C n^{s}\exp \left( \frac{n^{1-s}}{1-s} \right) .
\]
\end{cor}
The following lemma is really useful in the proof of Theorem \ref{tlcgrad}.
\begin{lem}[\cite{CG2015}]
\label{sumexppart} Let $\alpha,\beta $ be non-negative constants such that $0<\alpha<1$, and $\left( u_{n}\right)$, $\left( v_{n} \right)$ be two sequences defined for all $n \geq 1$ by
\begin{align*}
u_{n} & := \frac{c_{u}}{n^{\alpha}}, & v_{n}:=\frac{c_{v}}{n^{\beta}},
\end{align*}
with $c_{u},c_{v} > 0$. Thus, there is a positive constant $c_{0}$ such that for all $n \geq 1$,
\begin{align}
 & \sum_{k=1}^{n}e^{-\sum_{j=k}^{n}u_{j}}u_{k}v_{k}  = O \left( v_{n}\right) .
\end{align}
\end{lem}
Finally, we recall the following results, which enables us to upper bound the $L^{p}$ moments of a sum of random variables in normed vector spaces.
\begin{lem}[\cite{godichon2015}]
\label{lemsum} Let $Y_{1},...,Y_{n}$ be random variables taking values in a normed vector space such that for all positive constant $q$ and for all $k \geq 1$, $\mathbb{E}\left[ \left\| Y_{k} \right\|^{q} \right] < \infty$. Thus, for all constants $a_{1},...,a_{n}$ and for all integer $p$,
\begin{equation}
\mathbb{E}\left[ \left\| \sum_{k=1}^{n} a_{k}Y_{k} \right\|^{p} \right] \leq \left( \sum_{k=1}^{n} \left| a_{k} \right| \left( \mathbb{E}\left[ \left\| Y_{k} \right\|^{p} \right] \right)^{\frac{1}{p}} \right)^{p}.
\end{equation}
\end{lem}

\end{appendix}

\def\cprime{$'$}

\end{document}